\documentclass[12pt]{article}
\usepackage{amsthm}
\usepackage{amsfonts}
\usepackage{mathrsfs}
\usepackage{amsmath}
\usepackage[numbers,sort&compress]{natbib}
\newtheorem{thm}{Theorem}[section]
\newtheorem{prop}{Proposition}[section]
\newtheorem{lem}{Lemma}[section]
\newtheorem{rem}{Remark}[section]

\newtheorem{defn}{Definition}[section]
\newtheorem{exm}{Example}[section]
\newtheorem{ass}{Assumption}[section]

\numberwithin{equation}{section}
\allowdisplaybreaks[4]
\begin{document}
\date{}
\pagestyle{plain}
\title{Continuous-time Markov Decision Processes with Finite-horizon Expected Total Cost Criteria}
\author{Qingda Wei$^1$, \   Xian Chen$^2$\thanks{The corresponding author.} \thanks{Wei's  email:   weiqd@hqu.edu.cn;  Chen's email:  chenxian@amss.ac.cn}
  \   \\
  $^1$School of Economics and Finance, \\
  Huaqiao University, Quanzhou, 362021, P.R. China\\
  $^2$School of Mathematical Sciences,\\
Peking University, Beijing, 100871, P.R. China}
\date{}
\maketitle \underline{}
\begin{abstract}
This paper deals with the unconstrained and constrained cases for continuous-time Markov decision processes under the finite-horizon expected total cost criterion. The state space is denumerable and the transition and cost rates are allowed to be unbounded from above and from below. We give conditions for
the existence of optimal policies in the class of all randomized history-dependent policies. For the unconstrained case, using the analogue of the forward Kolmogorov equation in the form of conditional expectation, we show that the finite-horizon optimal value function is the unique solution to the optimality equation and obtain the existence of an optimal deterministic Markov policy. For the constrained case, employing  the technique of occupation measures, we first give an equivalent characterization of the occupation measures, and derive  that for each occupation measure generated by a randomized history-dependent policy, there exists an occupation measure
  generated by a randomized Markov policy equal to  it.
 Then using the compactness and convexity of the set of all occupation measures,  we obtain the existence of a constrained-optimal randomized Markov policy. Moreover, the constrained optimization problem is reformulated as a linear program,  and the strong duality between the linear program  and its dual program is  established.    Finally, a controlled birth and death system is used to illustrate our main results.


\end{abstract}
\setlength{\baselineskip}{0.25in}

\section{Introduction} \label{intro}
 Continuous-time Markov decision processes (CTMDPs) have been applied in many areas, such as queueing systems, epidemiology, and telecommunication; see, for instance, \cite{guo09,kit95,put} and the references therein. The optimality criteria for CTMDPs can be classified into the finite-horizon and  infinite-horizon criteria. As we can see in the existing literature,  the infinite-horizon criteria have been widely studied by many authors; see, for instance,  \cite{guo13,guo09,guo11,gs11,guo12,pi11,put} and their extensive references.
Comparing with the infinite-horizon criteria, there exist few works on the finite-horizon criteria for CTMDPs, whose treatment is more difficult than that of infinite-horizon criteria.
On the other hand, as we know, the finite-horizon criteria for discrete-time MDPs have found rich applications to portfolio investment,  inventory management, highway pavement maintenance, etc.;  see, for instance,  \cite{put,bau,dyn}.
In view of applications, sometimes it is more suitable to formulate the optimization models with finite-horizon criteria  than those  with infinite-horizon criteria.
We study the unconstrained and constrained cases for CTMDPs under the finite-horizon expected total cost criterion in this paper. Our main goals are as follows:
\begin{itemize}
\item[(i)]  Give conditions for the existence of optimal policies in the class of all randomized history-dependent policies for the unconstrained and constrained cases;
\item[(ii)] Formulate the constrained optimization problem  as a linear program;
\item[(iii)]  Establish the  strong duality between the primal linear program  and its dual program.
\end{itemize}

 For the unconstrained case,  we briefly describe the previous literature on the finite-horizon criteria for CTMDPs. The existence of a solution to the optimality equation is established in \cite{mill} for finite states and finite actions, in \cite{bau} for bounded transition rates and denumerable states, in \cite{gih,pli,yush} for bounded transition rates and Borel state spaces, in \cite{van} for unbounded transition rates and denumerable states, and in \cite{w} for unbounded transition rates and Borel spaces. It should be noted that all the aforementioned works restrict the discussions of the finite-horizon optimization problems  to  the class of all Markov policies. However, the decision-makers may make decisions basing on the past information. To consider  the past information,   the definition of a randomized history-dependent policy has been introduced in \cite{guo11,kit95,gs11,guo12,guo13,pi11} to study the infinite-horizon criteria for CTMDPs.

 In this paper we  discuss the finite-horizon criteria  with the  randomized history-dependent policies.  The state space is  denumerable and the action space is a Polish space. The transition and cost rates are allowed to be unbounded from above and from below. The dynamic programming approach
is used to prove the existence of   optimal policies  under the  suitable  conditions. We first give a new estimation of the weight function in the form of conditional expectation induced by  the randomized history-dependent policies, which extends the  results in \cite{guo09,pi11,guo11,gs11,guo13,guo12} (see Theorem \ref{thm4.1}). It should be mentioned that the extension of this  estimation is nontrivial. Then we derive
the analogue of the forward Kolmogorov equation in the form of conditional expectation by a technique of the dual predictable projection (see Theorem \ref{thm4.2}). Finally, applying the analogue of the forward Kolmogorov equation, we show that the finite-horizon optimal value function is the unique solution to the optimality equation and obtain the existence of an optimal deterministic Markov policy, which  extend the results in \cite{gih,pli,yush,w,van,bau,mill} from the class of all  Markov policies  to the more general class of all randomized history-dependent policies (see Theorem \ref{thm4.4}). It is worthy to point out that since the controlled state process does not have  the Markov property under any randomized history-dependent policy and the finite-horizon optimal value function includes a time variable,  the analyses are more difficult and complicated  than those of  the infinite-horizon criteria with the randomized history-dependent policies and the finite-horizon criteria with the Markov policies.
Moreover,  the fixed point theorem and uniformization techniques are inapplicable to the case of unbounded transition rates.

For the constrained case, the optimality criterion to be minimized is the finite-horizon expected total costs, and the constraints are imposed on the similar finite-horizon expected total costs. We employ the  convex analytic approach by introducing the occupation measures of the finite-horizon criteria. Under suitable conditions, we
give an equivalent characterization of the occupation measures and show that  the set of all occupation measures is convex, compact, and metrizable in the $w$-weak topology (see Theorem \ref{thm5.1}). From this equivalent characterization, we conclude that for each occupation measure generated by a randomized history-dependent policy, there exists an  occupation measure generated by
a randomized Markov policy   equal to it. Moreover, the constrained optimization problem can be reformulated as   a linear program.  Applying the Weierstrass theorem, we obtain the existence of a constrained-optimal randomized Markov policy (see Theorem \ref{thm5.2}). Finally, we develop the dual program of the linear program, and establish the strong duality between the primal linear program and its dual program  (see Theorem \ref{thm5.3}).

The rest of this paper is organized as follows. In Section
2,  we introduce the control model and   optimization problem. In Section 3, we give optimality  conditions for the existence
of optimal  policies and some preliminary results. The main results  for the unconstrained and constrained optimization problems are presented in Sections 4 and 5, respectively.
In Section  6, we  illustrate our main
results with   a controlled birth and death system.
\section{The control  model and optimization problem} \label{sec2}
The primitive data of the control model  in this paper are as follows:
$$\{S, A, (A(i), i\in S), q(j|i,a), c_0(i,a)(c_n(i,a), d_n, 1\leq n\leq N)\},$$
where the state space $S$ is assumed to be a denumerable set endowed with discrete topology  and the action space $A$ is assumed to be a Polish space with Borel $\sigma$-algebra $\mathcal{B}(A)$. $A(i)\in \mathcal{B}(A)$ denotes the set of  admissible actions when the state of the system is $i\in S$. Define $K:=\{(i,a)|\ i\in S, a\in A(i)\}$ which contains all the feasible state-action pairs.
The transition rates $q(j|i,a)$ are supposed to satisfy the following properties:
\begin{itemize}
\item For each fixed $i,j\in S$, $q(j|i,a)$ is measurable in $a\in A(i)$;
\item $q(j|i,a)\geq0$ for all $(i,a)\in K$ and  $j\neq i$;
\item $\sum_{j\in S} q(j|i,a)=0$ for all $(i,a)\in K$;
\item $q^*(i):=\sup_{a\in A(i)}|q(i|i,a)|<\infty$ for all $i\in S$.
\end{itemize}
Finally, the real-valued cost functions  $c_n(i,a)$ $(0\leq n\leq N)$ on $K$ are assumed to be measurable in $a\in A(i)$ for each $i\in S$ and the real numbers $d_n$ $(1\leq n\leq N)$ denote the constraints imposed on the finite-horizon expected total costs.

Let $S_{\infty}:=S\cup\{i_{\infty}\}$ with an isolated point $i_{\infty}\notin S$,  $\mathbb{R}_+:=(0,+\infty)$, $\mathbb{R}_+^0:=[0,+\infty)$,  $\Omega^0:=(S\times \mathbb{R}_+)^\infty$, and
$\Omega:=\Omega^0\cup\{(i_0, \theta_1, i_1,\ldots, \theta_{m-1}, i_{m-1}, \infty,i_{\infty},\infty, i_{\infty}, \ldots)|\ i_0\in S, \ i_l\in S,\ \theta_{l}\in \mathbb{R}_+  \  {\rm for \ each} \ 1\leq l\leq m-1,\  m\geq2\}$. Hence, we obtain a measurable space $(\Omega, \mathcal{F})$ in which  $\mathcal{F}$ is  the standard Borel $\sigma$-algebra.
Define the maps on $(\Omega,\mathcal{F})$ below: for each $\omega=(i_0, \theta_1, i_1,\ldots)\in \Omega$, let $T_0(\omega):=0,$ $X_0(\omega):=i_0$; for $m\geq 1$, let  $\Theta_m(\omega):=\theta_m,$  $X_m(\omega):=i_m$, $T_m(\omega):=\theta_1+\theta_2+\cdots+\theta_m$, $T_{\infty}(\omega):=\lim\limits_{m\to\infty}T_m(\omega)$, and
$$\xi_t(\omega):=\sum_{m\geq0}I_{\{T_m\leq t<T_{m+1}\}}i_m+I_{\{T_{\infty}\leq t\}}i_{\infty} \ \ {\rm for \ all}  \ t\geq0,$$
where $I_D$ denotes the indicator function of a set $D$. $\{T_m\}_{m\geq0}$ are the jump epochs, and $X_m$ is the state of the process $\{\xi_t,t\geq0\}$ on $[T_m, T_{m+1})$. Since we do not intend to consider the process after $T_{\infty}$, it is regarded to be absorbed in the state $i_{\infty}$. Hence, we write $q(i_{\infty}|i_{\infty},a_{\infty})=0$, where $a_{\infty}$ is an isolated point.
Moreover, we set  $A_{\infty}:=A\cup\{a_{\infty}\}$, $A(i_{\infty}):=\{a_{\infty}\}$,  $\mathcal{F}_t:=\sigma(\{T_m\leq s, X_m=j\}: j\in S, s\leq t, m\geq0)$ for all $t\geq0$, $\mathcal{F}_{s-}:=\bigvee_{0\leq t<s}\mathcal{F}_t$ (i.e.,  the smallest $\sigma$-algebra containing all the $\sigma$-algebras $\{\mathcal{F}_t, 0\leq t<s\}$), and  $\mathcal{P}:=\sigma(B\times\{0\}(B\in \mathcal{F}_0), B\times (s,\infty)(B\in\mathcal{F}_{s-}, s>0))$ which is the $\sigma$-algebra of predictable sets on $\Omega\times \mathbb{R}_+^0$ with respect to $\{\mathcal{F}_t\}_{t\geq0}$.

To define the optimality criteria, we introduce the definition of a policy below.
\begin{defn}
  {\rm A  $\mathcal{P}$-measurable transition probability $\pi(\cdot|\ \omega, t)$ on $(A_{\infty}, \mathcal{B}(A_{\infty}))$, concentrated on $A(\xi_{t-}(\omega))$, is called
  a randomized history-dependent policy. A policy is called randomized Markov if  there exists a kernel $\varphi$ on $A_{\infty}$ given $S_{\infty}\times \mathbb{R}_+^0$ such that $\pi(\cdot|\ \omega,t)=\varphi(\cdot| \xi_{t-}(\omega), t)$. A policy is called deterministic Markov if there exists a measurable function $f$ on  $S_{\infty}\times \mathbb{R}_+^0$ with $f(i,t)\in A(i)$ for all $(i,t)\in S_{\infty}\times \mathbb{R}_+^0$,  such that $\pi(\cdot|\ \omega, t)=\delta_{f(\xi_{t-}(\omega), t)}(\cdot)$, where $\delta_x(\cdot)$ is a Dirac measure concentrated at  $x$.
  }
\end{defn}
We denote by $\Pi$ the set of all randomized history-dependent policies, by $\Pi^M$ the set of
all randomized Markov policies,  and by  $\Pi^D$ the set of all deterministic Markov policies.

For any $\pi\in \Pi$, we define the random measure
\begin{eqnarray}\label{nu}
\nu^{\pi}(\omega, dt,j):=\left[\int_A\pi(da|\omega,t)q(j|\xi_{t-}(\omega),a)I_{\{j\neq \xi_{t-}(\omega)\}}\right]dt
\end{eqnarray}
for any  $j\in S$. Then, we have that this random measure is predictable, and $\nu^{\pi}(\omega, \{t\}\times S)=\nu^{\pi}(\omega, [T_{\infty},\infty)\times S)=0$. Hence, for any $\pi\in\Pi$ and
any initial distribution $\gamma$ on $S$, by Theorem 4.27 in \cite{kit95}, there exists a unique probability measure $P_{\gamma}^{\pi}$ on $(\Omega,\mathcal{F})$ such that $P_{\gamma}^{\pi}(\xi_0=i)=\gamma(i)$,  and with respect to $P_{\gamma}^{\pi}$,  $\nu^{\pi}$ is
the dual predictable projection of  random measure on $\mathbb{R}_+\times S$
\begin{eqnarray}\label{mu}
\mu(\omega, dt,i)=\sum_{m\geq1}I_{\{T_m<\infty\}}I_{\{X_m=i\}}\delta_{T_m}(dt).
\end{eqnarray}
Therefore, we obtain a stochastic basis $(\Omega, \mathcal{F}, \{\mathcal{F}_t\}_{t\geq0}, P_{\gamma}^{\pi})$, which is always assumed to be complete.  When $\gamma(j)=\delta_i(j)$ for all $j\in S$,  we write $P^{\pi}_{\gamma}$ as $P^{\pi}_i$. The expectation operators  with respect to $P^{\pi}_{\gamma}$ and $P^{\pi}_{i}$ are denoted as $E^{\pi}_{\gamma}$ and $E^{\pi}_{i}$, respectively.

For each initial state $i\in S$, $\pi\in\Pi$, and any fixed initial distribution $\gamma$ on $S$,   we define the finite-horizon expected total costs  from time 0 to the fixed terminal time $T>0$ as follows:
$$V_n(i,\pi):=E_i^{\pi}\left[\int_0^T\int_Ac_n(\xi_{t-},a)\pi(da|\omega,t)dt\right],$$
$$V_n(\pi):=E_{\gamma}^{\pi}\left[\int_0^T\int_Ac_n(\xi_{t-},a)\pi(da|\omega,t)dt\right]$$
for all $n=0,1,\ldots, N$,  provided that  the integrals are well defined.

 A policy $\pi^*\in \Pi$ is said to be finite-horizon optimal if $V_0(i,\pi^*)=\inf\limits_{\pi\in\Pi}V_0(i,\pi)$ for all $i\in S$.

Now we state the constrained optimization problem considered in this paper below:
\begin{eqnarray}\label{P}
  {\text  {\rm Minimize}} \ V_0(\pi) \ {\rm over} \ U:=\{\pi\in\Pi|\ V_n(\pi)\leq d_n, \ 1\leq n\leq N\}.
\end{eqnarray}
\begin{defn}
{\rm   A policy $\pi^*\in U$ is said to be constrained-optimal if $V_0(\pi^*)=\inf\limits_{\pi\in U}V_0(\pi)$.
}
\end{defn}

\section{Preliminaries}
In this section, we will give optimality conditions for the existence of optimal policies and some preliminary results  to prove our main results.

To avoid  the explosiveness of the process $\{\xi_t,t\geq0\}$, we need the following condition from \cite{pi11,guo09}.
\begin{ass}\label{ass3.1}
  There exist a weight  function $w\geq1$ on $S$, and constants $\rho_1>0$, $b_1\geq0$,  and $L>0$,  such that
\begin{itemize}
\item[{\rm (i)}] $\sum_{j\in S}w(j)q(j|i,a)\leq \rho_1 w(i)+b_1$ for all $(i,a)\in K$.
\item[{\rm (ii)}] $q^*(i)\leq Lw(i)$ for all $i\in S$.
\end{itemize}
\end{ass}
In order to guarantee the finiteness of finite-horizon expected total cost criteria, we also consider the conditions below, which are widely used in \cite{guo09,guo11,guo12,guo13,pi11,gs11}.
\begin{ass}\label{ass3.2}
\begin{itemize}
\item[{\rm (i)}] $\gamma(w):=\sum_{i\in S}w(i)\gamma(i)<\infty$, where $w$ comes from Assumption \ref{ass3.1}.
\item[{\rm (ii)}] There exists a constant $M>0$ such that $|c_n(i,a)|\leq Mw(i)$ for all $(i,a)\in K$ and $n=0,1,\ldots,N$.
\end{itemize}
\end{ass}
Under the above two assumptions, we have the following statements.
\begin{lem}\label{lem3.1}
Suppose that Assumptions \ref{ass3.1} and \ref{ass3.2} hold. Then for each $\pi\in\Pi$, we have
  \begin{itemize}
  \item[{\rm (a)}] $P_{\gamma}^{\pi}(T_{\infty}=\infty)=1$ and $P_{\gamma}^{\pi}(\xi_t\in S)=1$ for all $t\geq0$.
  \item[{\rm(b)}] $E_i^{\pi}[w(\xi_t)]\leq e^{\rho_1t}w(i)+\dfrac{b_1}{\rho_1}(e^{\rho_1 t}-1)$ for all $i\in S$ and $t\geq0$.
  \item[{\rm (c)}] $|V_n(i,\pi)|\leq MT\left[e^{\rho_1T}w(i)+\dfrac{b_1}{\rho_1}(e^{\rho_1T}-1)\right]$ for all $i\in S$ and $n=0,1,\ldots,N$.
  \item[{\rm(d)}] $|V_0(\pi)|\leq MT\left[e^{\rho_1T}\gamma(w)+\dfrac{b_1}{\rho_1}(e^{\rho_1T}-1)\right]$ for all $n=0,1,\ldots,N$.
  \end{itemize}
\end{lem}
\begin{proof}
 Parts (a) and (b) follow from Proposition 2.1 in \cite{pi11}.

(c) By Assumption \ref{ass3.2}(ii), we obtain
\begin{eqnarray*}
|V_n(i,\pi)|&\leq& E_i^{\pi}\left[\int_0^T\int_A|c_n(\xi_{t-},a)|\pi(da|\omega,t)dt\right]\\
&\leq& ME_i^{\pi}\left[\int_0^Tw(\xi_t)dt\right]\leq MT\left[e^{\rho_1T}w(i)+\dfrac{b_1}{\rho_1}(e^{\rho_1T}-1)\right]
\end{eqnarray*}
for all $i\in S$, $\pi\in\Pi$, and $n=0,1,\ldots,N$, where the last inequality follows from part (b).

(d) Part (d) follows immediately from part (c).
\end{proof}
In addition to Assumptions \ref{ass3.1} and \ref{ass3.2}, we also need the following conditions to ensure the existence of optimal policies.
\begin{ass}\label{ass3.3}
\begin{itemize}
  \item[{\rm (i)}] There exist  constants $\rho_2>0$, $\rho_3>0$,  $b_2\geq0$, and  $b_3\geq0$  such that
  \begin{eqnarray*}
     \sum_{i\in S}w^2(j)q(j|i,a)\leq \rho_2w^2(i)+b_2,  \ \
   and \ \ \sum_{j\in S}w^3(j)q(j|i,a)\leq \rho_3w^3(i)+b_3
  \end{eqnarray*}
for all $(i,a)\in K$, where $w$ comes from Assumption \ref{ass3.1}.

\item[{\rm (ii)}]  For each $i\in S$, the set $A(i)$ is compact.

\item[{\rm (iii)}] For each fixed $i,j\in S$ and $n=0,1,\ldots,N$,  the functions $c_n(i,a)$, $q(j|i,a)$ and $\sum_{k\in S}w(k)q(k|i,a)$ are  continuous in $a\in A(i)$.
\end{itemize}
\end{ass}
Finally, we give the following assertions which are used to prove our main results.

\begin{lem}\label{lem3.2}
Fix any $i\in S$, $\pi\in\Pi$, and $s\in \mathbb{R}_+$. Let $\mathcal{F}_{T_n}:=\sigma(X_m,T_m, 0\leq m\leq n)$.  For any integrable random variable $Z$ on $(\Omega,\mathcal{F},P_i^{\pi}),$ we have
$$E_i^{\pi}[Z|\ \mathcal{F}_{T_n}, s<T_{n+1}]I_{\{s\in [T_n,T_{n+1})\}}=\frac{E_i^{\pi}[ZI_{\{s<T_{n+1}\}}| \ \mathcal{F}_{T_n}]}{E_i^{\pi}[I_{\{s<T_{n+1}\}}| \ \mathcal{F}_{T_n}]}I_{\{s\in [T_n,T_{n+1})\}},$$
where   we make the convention that $\dfrac{0}{0}=0.$
\end{lem}
\begin{proof}
Since $\{s\in [T_n,T_{n+1})\}\in \sigma(\mathcal{F}_{T_n},\{s\in [T_n,T_{n+1})\}),$ we have
$$E_i^{\pi}[Z|\ \mathcal{F}_{T_n}, s<T_{n+1}]I_{\{s\in [T_n,T_{n+1})\}}=E_i^{\pi}[ZI_{\{s\in [T_n,T_{n+1})\}}|\ \mathcal{F}_{T_n}, s<T_{n+1}].$$
For any  $B\in \mathcal{F}_{T_n}$, straightforward calculations yield
\begin{eqnarray*}
  &&\int_{B\cap\{s<T_{n+1}\}}\frac{E_i^{\pi}[ZI_{\{s<T_{n+1}\}}| \ \mathcal{F}_{T_n}]}{E_i^{\pi}[I_{\{s<T_{n+1}\}}|\ \mathcal{F}_{T_n}]}I_{\{s\in [T_n,T_{n+1})\}}dP_i^{\pi}\\
&=&E_i^{\pi}\left[I_{\{s<T_{n+1}\}}\frac{E_i^{\pi}[ZI_{B\cap\{s<T_{n+1}\}\cap{\{T_n\leq s\}}}| \ \mathcal{F}_{T_n}]}{E_i^{\pi}[I_{\{s<T_{n+1}\}}|\ \mathcal{F}_{T_n}]}\right]\\
&=&E_i^{\pi}\left[E_i^{\pi}\left[I_{\{s<T_{n+1}\}}\frac{E_i^{\pi}[ZI_{B\cap\{s<T_{n+1}\}\cap{\{T_n\leq s\}}}|\ \mathcal{F}_{T_n}]}{E_i^{\pi}[I_{\{s<T_{n+1}\}}|\ \mathcal{F}_{T_n}]}\bigg|\ \mathcal{F}_{T_n}\right]\right]\\
&=&E_i^{\pi}\left[E_i^{\pi}[I_{\{s<T_{n+1}\}}| \ \mathcal{F}_{T_n}]\frac{E_i^{\pi}[ZI_{B\cap\{s<T_{n+1}\}\cap{\{T_n\leq s\}}}| \ \mathcal{F}_{T_n}]}{E_i^{\pi}[I_{\{s<T_{n+1}\}}|\ \mathcal{F}_{T_n}]}\right]\\
&=&E_i^{\pi}[ZI_{B\cap\{s<T_{n+1}\}\cap{\{T_n\leq s\}}}]\\
&=&\int_{B\cap\{s<T_{n+1}\}}ZI_{\{s\in [T_n,T_{n+1})\}}dP_i^{\pi},
\end{eqnarray*}
where the first equality is due to the fact that  $B\in \mathcal{F}_{T_n} $ and  $\{T_n\leq s\}\in\mathcal{F}_{T_n}$, and  $$\int_{B\cap\{T_{n+1}\leq s\}}\frac{E_i^{\pi}[ZI_{\{s<T_{n+1}\}}|\ \mathcal{F}_{T_n}]}{E_i^{\pi}[I_{\{s<T_{n+1}\}}|\ \mathcal{F}_{T_n}]}I_{\{s\in [T_n,T_{n+1})\}}dP_i^{\pi}=\int_{B\cap\{T_{n+1}\leq s\}}ZI_{\{s\in [T_n,T_{n+1})\}}dP_i^{\pi}=0.$$
By the monotone class theorem, we have that for any $C\in \sigma(\mathcal{F}_{T_n},\{s\in [T_n,T_{n+1})\}),$
$$\int_{C}\frac{E_i^{\pi}[ZI_{\{s<T_{n+1}\}}| \ \mathcal{F}_{T_n}]}{E_i^{\pi}[I_{\{s<T_{n+1}\}}|\ \mathcal{F}_{T_n}]}I_{\{s\in [T_n,T_{n+1})\}}dP_i^{\pi}=\int_{C}ZI_{\{s\in [T_n,T_{n+1})\}}dP_i^{\pi}.$$
Hence, the  assertion  follows from the definition of conditional expectation.
\end{proof}
Employing  Lemma \ref{lem3.2}, we have the following result.
\begin{thm}\label{thm4.1}
Suppose that  Assumption \ref{ass3.1}  holds. Then  for  any  $i\in S$,  $\pi\in\Pi$, and $s<t$, the following statements hold.
\begin{itemize}
\item[{\rm(a)}] $E_i^{\pi}[w(\xi_t)I_{\{t<T_{m+1}\}}|\ \mathcal{F}_s]\leq (e^{\rho_1(t-s)}+1)w(\xi_s)+\dfrac{b_1}{\rho_1}(e^{\rho_1(t-s)}-1)$ for all $m=0,1,\ldots$.
\item[{\rm (b)}] $E_i^{\pi}[w(\xi_t)|\ \xi_s]\leq (e^{\rho_1(t-s)}+1)w(\xi_s)+\dfrac{b_1}{\rho_1}(e^{\rho_1(t-s)}-1).$
\end{itemize}
\end{thm}
\begin{proof}
(a) Because $s<t$, and  the sets $\{T_n\leq s\}$, $\{ T_\infty\leq s\}$ are in $\mathcal{F}_s$, we have
\begin{eqnarray}\label{3-0}
&&E_i^{\pi}[w(\xi_t)I_{\{t<T_{m+1}\}}|\ \mathcal{F}_s]\nonumber\\
&=&\sum_{n=0}^{\infty}E_i^{\pi}[w(\xi_t)I_{\{t<T_{m+1}\}}|\ \mathcal{F}_s]I_{\{s\in [T_n,T_{n+1})\}}+E_i^{\pi}[w(\xi_t)I_{\{t<T_{m+1}\}}|\ \mathcal{F}_s]I_{\{ T_\infty\leq s\}}\nonumber\\
&=&\sum_{n=0}^{m}E_i^{\pi}[w(\xi_t)I_{\{t<T_{m+1}\}}|\ \mathcal{F}_s]I_{\{s\in [T_n,T_{n+1})\}}.
\end{eqnarray}
Below we fix $n\in \{0,1,\ldots,m\}$, and set $\mathcal{F}_{T_n}:=\sigma(X_l,T_l, 0\leq l\leq n)$.
Since $\{s\in [T_n,T_{n+1})\}\in \mathcal{F}_s \cap \sigma(\mathcal{F}_{T_n},\{s\in [T_n,T_{n+1})\})$ and
$\{s\in [T_n,T_{n+1})\}\cap \mathcal{F}_s=\{s\in [T_n,T_{n+1})\}\cap \sigma(\mathcal{F}_{T_n},\{s\in [T_n,T_{n+1})\}),$ by Lemma 6.2 in \cite{Kall},  we obtain
\begin{align}\label{3-1}
E_i^{\pi}[w(\xi_t)I_{\{t<T_{m+1}\}}|\ \mathcal{F}_s]I_{\{s\in [T_n,T_{n+1})\}}
=&E_i^{\pi}[w(\xi_t)I_{\{t<T_{m+1}\}}|\ \mathcal{F}_{T_n},s\in [T_n,T_{n+1})]I_{\{s\in [T_n,T_{n+1})\}}\nonumber\\
=&E_i^{\pi}[w(\xi_t)I_{\{t<T_{m+1}\}}|\ \mathcal{F}_{T_n},s<T_{n+1}]I_{\{s\in [T_n,T_{n+1})\}}\nonumber\\
=&\frac{E_i^{\pi}[w(\xi_t)I_{\{t<T_{m+1}\}\cap\{ s<T_{n+1}\}}|\ \mathcal{F}_{T_n}]}{E_i^{\pi}[I_{\{s<T_{n+1}\}}]|\ \mathcal{F}_{T_n}]}I_{\{s\in [T_n,T_{n+1})\}},
\end{align}
where the second equality holds because $\{T_n\leq s\}\in \mathcal{F}_{T_n},$ and the third one follows from Lemma \ref{lem3.2}.
   Moreover,  direct computations yield
\begin{align*}
&E_i^{\pi}[w(\xi_t)I_{\{t<T_{m+1}\}\cap\{ s<T_{n+1}\}}|\ \mathcal{F}_{T_n}]I_{\{s\in [T_n,T_{n+1})\}}\\
=&E_i^{\pi}[E_i^{\pi}[w(\xi_t)I_{\{t<T_{m+1}\}\cap\{ s<T_{n+1}\}}|\ \mathcal{F}_{T_{n+1}}]|\ \mathcal{F}_{T_n}]I_{\{s\in [T_n,T_{n+1})\}}\\
=&E_i^{\pi}[E_i^{\pi}[w(\xi_t)I_{\{t<T_{m+1}\}}|\ \mathcal{F}_{T_{n+1}}]I_{\{T_n\leq s<T_{n+1}\}}|\ \mathcal{F}_{T_n}]I_{\{s\in [T_n,T_{n+1})\}}\\
\leq&E_i^{\pi}\left[\left\{I_{\{T_{n+1}\leq t\}}h(T_{n+1},X_{n+1}, t)
+\sum_{l=1}^{n+1}I_{\{T_{l-1}\leq t<T_l\}}w(X_{l-1})\right\}I_{\{ T_n\leq s<T_{n+1}\}}\bigg|\ \mathcal{F}_{T_n}\right]I_{\{s\in [T_n,T_{n+1})\}}\\
=&E_i^{\pi}\left[ I_{\{T_{n+1}\leq t\}}I_{\{ T_n\leq s<T_{n+1}\}}h(T_{n+1},X_{n+1}, t)+I_{\{T_{n}\leq t<T_{n+1}\}}I_{\{ T_n\leq s<T_{n+1}\}}w(X_{n})  |\ \mathcal{F}_{T_n}\right]I_{\{s\in [T_n,T_{n+1})\}}\\
=&E_i^{\pi}\left[ I_{\{s<T_{n+1}\leq t\}}h(T_{n+1},X_{n+1},t)+I_{\{t<T_{n+1}\}}w(X_{n})|\ \mathcal{F}_{T_n}\right]I_{\{s\in [T_n,T_{n+1})\}},
\end{align*}
where $h(\underline{t},i, \overline{t}):=e^{\rho_1(\overline{t}-\underline{t})}w(i)+\dfrac{b_1}{\rho_1}(e^{\rho_1(\overline{t}-\underline{t})}-1)$ for all $i\in S$, $0\leq\underline{t}\leq\overline{t}$, and the inequality follows from equality (36) in \cite{guo11}.
Note that $I_{\{t<T_{n+1}\}}\leq I_{\{s<T_{n+1}\}}$. On one hand, we get
\begin{align*}
  \frac{E_i^{\pi}\left[I_{\{t<T_{n+1}\}}w(X_{n})|\ \mathcal{F}_{T_n}\right]I_{\{s\in [T_n,T_{n+1})\}}}{E_i^{\pi}[I_{\{s<T_{n+1}\}}|\ \mathcal{F}_{T_n}]}
=&\frac{E_i^{\pi}\left[I_{\{t<T_{n+1}\}}|\ \mathcal{F}_{T_n}\right]I_{\{s\in [T_n,T_{n+1})\}}}{E_i^{\pi}[I_{\{s<T_{n+1}\}}|\ \mathcal{F}_{T_n}]}w(X_{n})\\
\leq & w(X_{n})I_{\{s\in [T_n,T_{n+1})\}} .
\end{align*}
On the other hand, it follows from  Lemma 3.3 in \cite{jacod} that  the function
$$\Lambda^{\pi}(j|\omega, t):=\int_A\pi(da|\omega,t)q(j|\xi_{t-}(\omega),a)I_{\{j\neq \xi_{t-}(\omega)\}}$$
has the representation below:
$$\Lambda^{\pi}(j|\omega,t)=I_{\{t=0\}}\Lambda^0(j|i_0)+\sum_{l\geq0}I_{\{T_l(\omega
)< t\leq T_{l+1}(\omega
)\}}\Lambda^l(j|i_0,\theta_1, i_1,\ldots, \theta_l, i_l, t-T_l(\omega
)),$$
where $\Lambda^l(j|i_0,\theta_1, i_1,\ldots, \theta_l, i_l, \widetilde{t})$ are some nonnegative nonrandom measurable functions. Employing the construction of the measure $P_{i}^{\pi}$ (see \cite{guo11} for details), we have
\begin{align*}&\frac{E_i^{\pi}\left[I_{\{s<T_{n+1}\leq t\}}h(T_{n+1},X_{n+1},t)|\ \mathcal{F}_{T_n}\right]I_{\{s\in [T_n,T_{n+1})\}}}{E_i^{\pi}[I_{\{s<T_{n+1}\}}|\ \mathcal{F}_{T_n}]}\\
=&\int_{s-T_n}^{t-T_n}\bigg\{e^{-\int_{0}^{u}\sum\limits_{j\neq X_n}\Lambda^{n}(j|X_0,\Theta_1, X_1,\ldots, \Theta_n, X_n,v)dv}\\&\times\sum\limits_{k\neq X_n}h(T_n+u,k,t)\Lambda^{n}(k|X_0,\Theta_1, X_1,\ldots, \Theta_n, X_n,u)\bigg\}du\\ &\times\left(e^{-\int_{0}^{s-T_n}\sum\limits_{j\neq X_n}\Lambda^{n}(j|X_0,\Theta_1, X_1,\ldots, \Theta_n, X_n,v)dv}\right)^{-1}I_{\{s\in [T_n,T_{n+1})\}}\\
=& \int_{s-T_n}^{t-T_n}\bigg\{e^{-\int_{s-T_n}^{u}\sum\limits_{j\neq X_n}\Lambda^{n}(j|X_0,\Theta_1, X_1,\ldots, \Theta_n, X_n,v)dv}\\&\times\sum\limits_{k\neq X_n}h(T_n+u,k,t)\Lambda^{n}(k|X_0,\Theta_1, X_1,\ldots, \Theta_n, X_n, u)\bigg\}duI_{\{s\in [T_n,T_{n+1})\}} \\
\leq & h(s-T_n,X_n,t-T_n)I_{\{s\in [T_n,T_{n+1})\}}= h(0,X_n,t-s)I_{\{s\in [T_n,T_{n+1})\}},
\end{align*}

where the inequality follows from Lemma A.1 in \cite{guo11}.
Hence, we obtain
\begin{eqnarray}\label{3-2}
\frac{E_i^{\pi}[w(\xi_t)I_{\{t<T_{m+1}\}\cap\{ s<T_{n+1}\}}|\ \mathcal{F}_{T_n}]}{E_i^{\pi}[I_{\{s<T_{n+1}\}}]|\ \mathcal{F}_{T_n}]}I_{\{s\in [T_n,T_{n+1})\}}
\leq [h(0,X_n,t-s)+w(X_{n})]I_{\{s\in [T_n,T_{n+1})\}}.
\end{eqnarray}
Therefore, by (\ref{3-0})-(\ref{3-2}),  we have
\begin{eqnarray*}
  E_i^{\pi}[w(\xi_t)I_{\{t<T_{m+1}\}}|\ \mathcal{F}_s]&\leq&
\sum_{n=0}^{m}[h(0,X_n,t-s)+\omega(X_{n})]I_{\{s\in [T_n,T_{n+1})\}}\\
&=&\sum_{n=0}^{m}[h(0,\xi_s,t-s)+\omega(\xi_s)]I_{\{s\in [T_n,T_{n+1})\}}\\
&\leq&h(0,\xi_s,t-s)+\omega(\xi_s).
\end{eqnarray*}

(b) By part (a), we have
\begin{eqnarray}\label{3-4}
  E_i^{\pi}[w(\xi_t)I_{\{t<T_{m+1}\}}|\ \xi_s]\leq (e^{\rho_1(t-s)}+1)w(\xi_s)+\frac{b_1}{\rho_1}(e^{\rho_1(t-s)}-1)
\end{eqnarray}
for all $m=0,1,2,\ldots$. Moreover, it follows from  Proposition 2.1 in \cite{pi11} that
$$P_i^{\pi}\left(\lim\limits_{m\to\infty}w(\xi_t)I_{\{t<T_{m+1}\}}=w(\xi_t)\right)=1.$$
Hence, the  assertion follows from (\ref{3-4}), Lemma \ref{lem3.1}(a)  and the dominated convergence theorem of conditional expectation.
This completes the proof of the theorem.
\end{proof}
\begin{rem}
{\rm Theorem \ref{thm4.1} presents
  a \emph{new} estimation on the so-called weight function $w$ in the form of conditional expectation induced by  the randomized history-dependent policies, which generalizes the estimation in  the case of conditional expectation induced by the Markov policies in \cite{guo09} and the case of expectation induced by the randomized history-dependent policies in \cite{pi11,guo11,gs11,guo13,guo12}.
 Moreover, since the   state process does not have the Markov property under any randomized history-dependent policy,  the technique in  \cite{guo09} is inapplicable here.
  }
\end{rem}
The following assertion extends the analogue of the forward Kolmogorov equation in \cite{guo11,gs11,guo12,guo13,pi11} to that in the form of conditional expectation.
\begin{thm}\label{thm4.2}
Suppose that Assumption \ref{ass3.1} holds. Then for any $i\in S$, $B\subseteq S$,  $\pi\in\Pi$, and $s<t$, we have
\begin{eqnarray*}
P_{i}^{\pi}(\xi_t\in B|\ \mathcal{F}_s)&=&I_{\{\xi_s\in B\}}+E_{i}^{\pi}\left[\int_{s}^{t}\int_{A}q(B|\xi_{u},a)\pi(da|\omega,u)du\bigg|\ \mathcal{F}_s\right].
\end{eqnarray*}
\end{thm}
\begin{proof}
Fix any $i\in S$, $B\subseteq S$, $\pi\in\Pi$,  and $s\in \mathbb{R}^0_+$. Let
$\nu_1^{\pi}(\omega, dt,j):=I_{(s,\infty)}(t)\nu^{\pi}(\omega,dt,j)$ and $\mu_1(\omega,dt,j):=I_{(s,\infty)}(t)\mu(\omega,dt,j)$, where the random measures $\nu^{\pi}$ and $\mu$ are as in (\ref{nu}) and (\ref{mu}), respectively.  Since $Y(\omega,t):=I_{(s,\infty)}(t)$ is a predictable process,
the dual predictable projection of $\mu_1$ is $\nu_1^{\pi}$.
Define the random measures below:
 $$\tilde{\mu}_1(\omega,dt,j):=\sum_{m\geq 1}I_{\{T_m<\infty\}}I_{\{X_{m-1}=j\}}I_{(s,\infty)}(t)\delta_{T_m}(dt),$$
 and
 $$\tilde{\nu}_1^{\pi}(\omega, dt,j):=I_{(s,\infty)}(t)\int_A (-q(\xi_{t-}|\xi_{t-},a))\pi(da|\omega, t)\delta_{\xi_{t-}}(j)dt.$$
 Then,
by Lemma 4 in \cite{kit85}, we have that $\tilde{\nu}^{\pi}_1$ is the dual predictable projection of $\tilde{\mu}_1$ with respect to $P_i^{\pi}$. From the proof of Theorem 3.1 in \cite{guo11}, we have
$$E_{i}^{\pi}[\mu_1((0,t],B)]<\infty  \   {\rm and} \  E_{i}^{\pi}[\tilde{\mu}_1((0,t],B)]<\infty \ \ {\rm for \ all}\ t>s.$$
Obviously, we have   the following equation
$$ I_{\{\xi_t\in B\}}=I_{\{\xi_s\in B\}}+\mu_1((0,t],B)-\tilde{\mu}_1((0,t],B).$$
Hence, taking conditional expectation in the both sides of the last equation,  we obtain
\begin{eqnarray*}
E_{i}^{\pi}[I_{\{\xi_t\in B\}}|\ \mathcal{F}_s]&=&I_{\{\xi_s\in B\}}+ E_{i}^{\pi}[\mu_1((0,t],B)|\ \mathcal{F}_s]-E_{i}^{\pi}[\tilde{\mu}_1((0,t],B)|\ \mathcal{F}_s]\\
&=&I_{\{\xi_s\in B\}}+ E_{i}^{\pi}[\nu^{\pi}_1((0,t],B)|\ \mathcal{F}_s]-E_{i}^{\pi}[\tilde{\nu}^{\pi}_1((0,t],B)|\ \mathcal{F}_s]\\
&=&I_{\{\xi_s\in B\}}+ E_{i}^{\pi}\left[\int_{s}^{t}\int_{A}q(B\setminus\{\xi_u\}|\xi_{u},a)\pi(da|\omega,u)du\bigg|\ \mathcal{F}_s\right]\\
&&+E_{i}^{\pi}\left[\int_{s}^{t}\int_{A}q(\xi_u|\xi_u,a)\pi(da|\omega,u)I_{\{\xi_u\in B\}}du\bigg|\ \mathcal{F}_s\right],
\end{eqnarray*}
where the second equality is due to the fact that $\mu_1((0,t],B)-\nu^{\pi}_1((0,t],B)$ and  $\tilde{\mu}_1((0,t],B)-\tilde{\nu}^{\pi}_1((0,t],B)$ are martingales.
\end{proof}
\section{Dynamic programming for the unconstrained case}
In this section, we will use the dynamic programming approach to show the existence of   optimal policies.  To this end, we introduce the following notation.

For any  $s\in [0,T]$,  a function $g$ defined on $S\times [s,T]$ is said to be \emph{$[s,T]$-uniformly $w^2$-bounded} if it is measurable    and there exists a constant $\tilde{M}>0$ such that $|g(i,t)|\leq \tilde{M}w^2(i)$ for all $(i,t)\in S\times[s,T].$

For any  $i,j\in S$,  $\pi\in\Pi$, and $s\in [0,T]$,
define the set
\begin{align}\label{4.4}
  \mathcal{H}_{s,j}^{i,\pi}:=\bigg\{g:  \ & g \ \mbox{is $[s,T]$-uniformly $w^2$-bounded and satisfies}\nonumber\\
&E_i^{\pi}\left[\int_s^T\int_A\sum_{k\in S}\bigg[\int_t^Tg(k,v)dv\bigg]q(k|\xi_t,a)\pi(da|\omega,t)dt\bigg|\ \xi_s=j\right]\nonumber\\
=&E_i^{\pi}\left[\int_s^Tg(\xi_t, t)dt\bigg|\ \xi_s=j\right]-\int_s^Tg(j, t)dt\bigg\}.
\end{align}

For each initial state $i\in S$, $\pi\in\Pi$, and $s\in [0,T]$,  the expected total cost from $s\geq0$ to the terminal time $T>0$ and the corresponding optimal value function are defined as
$$U(i,j,s,\pi):=E_i^{\pi}\left[\int_s^T\int_Ac_0(\xi_{t-},a)\pi(da|\omega,t)dt\bigg|\ \xi_s=j\right] \ \ {\rm and} \ \ U^*(i,j,s):=\inf_{\pi\in\Pi}U(i,j,s, \pi)$$
for all $j\in S$,
respectively.
In particular, when $s=0$, we have $U(i,i,0,\pi)=V_0(i,\pi)$.

Then we have the following \emph{new} assertion on the property of $\mathcal{H}_{s,j}^{i,\pi}$.
\begin{thm}\label{thm4.3}
Under  Assumptions  \ref{ass3.1} and \ref{ass3.3}(i),   the following statement holds:
for any  $i,j\in S$,  $\pi\in\Pi$ and $s\in [0,T]$, the set $\mathcal{H}_{s,j}^{i,\pi}$ in  {\rm(\ref{4.4})} contains all $[s,T]$-uniformly $w^2$-bounded functions.
\end{thm}
\begin{proof}
 Fix any $i,j\in S$, $\pi\in\Pi$, and $s\in[0,T]$. By Proposition 2.1 in \cite{pi11},  we obtain
\begin{eqnarray*}
&&E_i^{\pi}\left[\int_s^T\int_A\sum_{k\in S}w^2(k)|q(k|\xi_t,a)|\pi(da|\omega,t)dt\right]\\
 &\leq& (\rho_2+b_2+2L)E_i^{\pi}\left[\int_s^Tw^3(\xi_t)dt\right]\\
 &\leq&(\rho_2+b_2+2L) \int_s^T\left[e^{\rho_3 t}w^3(i)+\frac{b_3}{\rho_3}(e^{\rho_3t}-1)\right]dt\\
  &\leq&(\rho_2+b_2+2L)T\bigg[e^{\rho_3T}w^3(i)+\frac{b_3}{\rho_3}(e^{\rho_3T}-1)\bigg].
\end{eqnarray*}
If $g$  is $[s,T]$-uniformly $w^2$-bounded, from the last inequality, we have
\begin{eqnarray*}
E_i^{\pi}\left[\int_s^T\int_A\sum_{k\in S}\bigg|\int_t^Tg(k,v)dv\bigg||q(k|\xi_t,a)|\pi(da|\omega,t)dt\bigg|\ \xi_s=j\right]<\infty.
\end{eqnarray*}
Hence, $E_i^{\pi}\left[\int_s^T\int_A\sum_{k\in S}\left[\int_t^Tg(k,v)dv\right]q(k|\xi_t,a)\pi(da|\omega,t)dt\big|\ \xi_s=j\right]$ is well defined. Using the similar  arguments, we see that  $E_i^{\pi}\left[\int_s^Tg(\xi_t, t)dt\big|\ \xi_s=j\right]$ is well defined.

Let  $\mathcal{C}:=\{B\times [t_*,t^*]: B\subseteq S,  \ s\leq t_*\leq t^*\leq T\}.$ Then, it is obvious  that $\mathcal{C}$ is a $\pi$-system and $S\times [s,T]\in \mathcal{C}$.  Below we will use the monotone class theorem to show that $\mathcal{H}_{s,j}^{i,\pi}$ contains all the bounded measurable functions on $S\times [s,T]$.

(i) For any $B\times[t_*,t^*]\in \mathcal{C}$,  we will  show that $I_B(k)I_{[t_*,t^*]}(t)\in \mathcal{H}_{s,j}^{i,\pi}.$  Set $t_1\vee t_2:=\max\{t_1,t_2\}$ and  $t_1\wedge t_2:=\min\{t_1,t_2\}$.  Direct calculations yield
\begin{eqnarray*}
&&E_i^{\pi}\left[\int_s^T\int_A\sum_{k\in S}\bigg[\int_t^TI_B(k)I_{[t_*,t^*]}(v)dv\bigg]q(k|\xi_t,a)\pi(da|\omega,t)dt\bigg|\ \xi_s=j\right]\\
&=&E_i^{\pi}\left[\int_s^T[t^*-(t_*\vee t)\wedge t^*]\int_Aq(B|\xi_t,a)\pi(da|\omega,t)dt\bigg|\ \xi_s=j\right] \\
&=&E_i^{\pi}\left[\int_s^{t_*}[t^*-(t_*\vee t)\wedge t^*]\int_Aq(B|\xi_t,a)\pi(da|\omega,t)dt\bigg|\ \xi_s=j\right]\\
&&+E_i^{\pi}\left[\int_{t_*}^{t^*}[t^*-(t_*\vee t)\wedge t^*]\int_Aq(B|\xi_t,a)\pi(da|\omega,t)dt\bigg|\ \xi_s=j\right]\\
&=&(t^*- t_*)E_i^{\pi}\left[\int_s^{t_*}\int_Aq(B|\xi_t,a)\pi(da|\omega,t)dt\bigg|\ \xi_s=j\right]\\
&&+E_i^{\pi}\left[\int_{t_*}^{t^*}(t^*-t)\int_Aq(B|\xi_t,a)\pi(da|\omega,t)dt\bigg|\ \xi_s=j\right]\\
&=&E_i^{\pi}\left[\int_{t_*}^{t^*}\int_s^{t}\int_Aq(B|\xi_v,a)\pi(da|\omega,v)dvdt\bigg|\ \xi_s=j\right]\\
&=&\int_{t_*}^{t^*}E_i^{\pi}\left[\int_s^{t}\int_Aq(B|\xi_v,a)\pi(da|\omega,v)dv\bigg|\ \xi_s=j\right]dt\\
&=&\int_{t_*}^{t^*}E_i^{\pi}[I_B(\xi_t)|\ \xi_s=j]dt-\int_{t_*}^{t^*}I_B(j)dt\\
&=&E_{i}^{\pi}\left[\int_s^TI_B(\xi_t)I_{[t_*,t^*]}(t)dt\bigg|\ \xi_s=j\right]-\int_s^TI_B(j)I_{[t_*,t^*]}(t)dt,
\end{eqnarray*}
where the fourth equality is due to  the integration by parts,   and the sixth one follows from  Theorem \ref{thm4.2}.  Hence, we have $I_B(k)I_{[t_*,t^*]}(t)\in \mathcal{H}_{s,j}^{i,\pi}.$

(ii) If $0\leq g_n\in \mathcal{H}_{s,j}^{i,\pi}$ $(n=1,2,\ldots)$,  $g_n\uparrow g_0$ and $g_0$ is bounded,  applying the monotone convergence theorem, we have $g_0\in \mathcal{H}_{s,j}^{i,\pi}$.

Obviously,   $\mathcal{H}_{s,j}^{i,\pi}$ is a linear space. Hence, (i), (ii) and the monotone class theorem give that $\mathcal{H}_{s,j}^{i,\pi}$ contains all measurable bounded functions on $S\times [s,T]$. Therefore,  the desired assertion follows from   the same technique of Lemma 3.7 in \cite{w}.
\end{proof}
Using Theorem \ref{thm4.3}, we obtain the  main result on the existence of finite-horizon optimal policies for the case of unbounded transition and cost rates below.
\begin{thm}\label{thm4.4}
  Suppose that Assumptions  \ref{ass3.1}, \ref{ass3.2}(ii), and \ref{ass3.3} hold. Then  we have
  \begin{itemize}
    \item[{\rm(a)}] For any given $i\in S$, the   function $U^*(i,\cdot)$ is the unique solution in $B_{w}(S\times [0,T])$ to the following equation: for each $j\in S$ and $s\in [0,T]$,
        \begin{eqnarray}\label{4.6}
          g(j,s)=\int_s^T\inf_{a\in A(j)}\left\{c_0(j,a)+\sum_{k\in S}g(k,t)q(k|j,a)\right\}dt,
        \end{eqnarray}
        where $B_{w}(S\times [0,T])$ denotes the set of all real-valued measurable functions $g$ on $S\times [0,T]$ with
      $\sup_{j\in S}\sup_{s\in[0,T]}\frac{|g(j,s)|}{w(j)}<\infty.$
    \item[{\rm (b)}] There exists an  optimal deterministic Markov policy  $\pi^*_T$ (depending on $T$).
  \end{itemize}
\end{thm}
\begin{proof}
(a) Fix any $i\in S$, $\pi\in\Pi$, and $s\in [0,T]$. Then,
direct calculations give
\begin{eqnarray*}
  |U(i,j,s,\pi)|&\leq& E_i^{\pi}\left[\int_s^T\int_A|c_0(\xi_{t-},a)|\pi(da|\omega,t)dt\bigg|\ \xi_s=j\right]\\
  &\leq& ME_i^{\pi}\left[\int_s^Tw(\xi_t)dt\bigg|\ \xi_s=j\right]\\
  &=&M\int_s^TE_i^{\pi}[w(\xi_t)|\ \xi_s=j]dt\\
  &\leq& M\int_s^T\left[(e^{\rho_1(t-s)}+1)w(j)+\frac{b_1}{\rho_1}(e^{\rho_1(t-s)}-1)\right]dt\\
  &\leq& MT\left[(e^{\rho_1T}+1)+\frac{b_1}{\rho_1}(e^{\rho_1T}-1)\right]w(j)
\end{eqnarray*}
for all $j\in S$,
where the second inequality is due to Assumption \ref{ass3.2}(ii), and the  fourth one follows from  Theorem \ref{thm4.1}. Hence, we have
$$\sup_{j\in S}\sup_{s\in [0,T]}\frac{|U^*(i,j,s)|}{w(j)}\leq MT\left[(e^{\rho_1T}+1)+\frac{b_1}{\rho_1}(e^{\rho_1T}-1)\right]<\infty.$$
Moreover, it follows from the proof of  Theorem 4.1 in \cite{w}  that there exists a function $g$ on $S\times [0,T]$ satisfying  (\ref{4.6}), and that for each $j\in S$, the partial  derivative of $g$ with respect to the second variable $s$ exists, denoted by $\frac{\partial g}{\partial s}$. Thus, we obtain
\begin{eqnarray}\label{4.7}
-\frac{\partial g}{\partial s}(j,s)=\inf_{a\in A(j)}\left\{c_0(j,a)+\sum_{k\in S}g(k,s)q(k|j,a)\right\}\ \ {\rm for \ all} \ j\in S.
\end{eqnarray}
The measurable selection theorem in \cite[p.50]{one99} and Assumption \ref{ass3.3} imply that for each $j\in S$, the function  $\frac{\partial g}{\partial s}(j,\cdot)$ on $[0,T]$ is measurable. Set $M^*:=\sup_{j\in S}\sup_{s\in[0,T]}\frac{|g(j,s)|}{w(j)}$.  Then it follows from (\ref{4.7}),  Assumptions \ref{ass3.1} and \ref{ass3.2}(ii) that
$$\sup_{j\in S}\sup_{s\in [0,T]}\frac{|\frac{\partial g}{\partial s}(j,s)|}{w^2(j)}\leq M+M^*(\rho_1+b_1+2L).$$
Hence,  $\frac{\partial g}{\partial s}$ is a $[0,T]$-uniformly $w^2$-bounded function. Note that $g(j,T)=0$ for all $j\in S$. Therefore, by Theorem \ref{thm4.3},   we have that for each $j\in S$,
\begin{eqnarray}\label{4.8}
  &&E_i^{\pi}\left[\int_s^T\int_A\sum_{k\in S}\bigg[\int_t^T\frac{\partial g}{\partial v}(k,v)dv\bigg]q(k|\xi_t,a)\pi(da|\omega,t)dt\bigg|\ \xi_s=j\right]\nonumber\\
  &=&E_i^{\pi}\left[\int_s^T\frac{\partial g}{\partial t}(\xi_t, t)dt\bigg|\ \xi_s=j\right]-\int_s^T\frac{\partial g}{\partial t}(j, t)dt\nonumber\\
  &=&-E_i^{\pi}\left[\int_s^T\int_A\sum_{k\in S}g(k,t)q(k|\xi_t,a)\pi(da|\omega,t)dt\bigg|\ \xi_s=j\right].
\end{eqnarray}
On one hand, by (\ref{4.7}), we obtain
$$-\int_s^T\frac{\partial g}{\partial t}(\xi_t,t)dt\leq \int_s^T\int_Ac_0(\xi_t,a)\pi(da|\omega,t)dt+\int_s^T\int_A\sum_{k\in S}g(k,t)q(k|\xi_t,a)\pi(da|\omega,t)dt.$$
Taking the conditional expectation in the both sides of the last inequality, by (\ref{4.8}),  we have
\begin{eqnarray*}
&&-E_i^{\pi}\left[\int_s^T\frac{\partial g}{\partial t}(\xi_t,t)dt\bigg|\ \xi_s=j\right]\\
&\leq& U(i,j,s,\pi)+E_i^{\pi}\left[\int_s^T\int_A\sum_{k\in S}g(k,t)q(k|\xi_t,a)\pi(da|\omega,t)dt\bigg|\ \xi_s=j\right]\\
&=&U(i,j,s,\pi)+\int_s^T\frac{\partial g}{\partial t}(j, t)dt-E_i^{\pi}\left[\int_s^T\frac{\partial g}{\partial t}(\xi_t, t)dt\bigg|\ \xi_s=j\right]\\
&=&U(i,j,s,\pi)-g(j,s)-E_i^{\pi}\left[\int_s^T\frac{\partial g}{\partial t}(\xi_t, t)dt\bigg|\ \xi_s=j\right],
\end{eqnarray*}
which implies $g(j,s)\leq U(i,j,s,\pi)$ for all $j\in S$. By the arbitrariness of $\pi$,  we obtain
\begin{eqnarray}\label{4.9}
g(j,s)\leq U^*(i,j,s) \ \ {\rm for \ all} \ j\in S \ {\rm and} \ s\in[0,T].
\end{eqnarray}
On the other hand, it follows from Assumption \ref{ass3.3} and the measurable selection theorem in \cite[p.50]{one99} that there exists a measurable function $f^*$ on $S\times [0,T]$ satisfying  $f^*(j,s)\in A(j)$ and
\begin{eqnarray*}
-\frac{\partial g}{\partial s}(j,s)=c_0(i,f^*(j,s))+\sum_{k\in S}g(k,s)q(k|j,f^*(j,s))
\end{eqnarray*}
for   all $j\in S$ and $s\in [0,T]$.  Let $\pi^*(\cdot|\ \omega, s):=\delta_{f^*(\xi_{s-}(\omega), s)}(\cdot)$. Then, combining the last equality and following the similar arguments of (\ref{4.9}), we have
\begin{eqnarray}\label{4.10}
g(j,s)= U(i,j,s,\pi^*)\geq U^*(i,j,s)\ \ {\rm for \ all} \ j\in S \ {\rm and} \ s\in[0,T].
\end{eqnarray}
Hence, the statement follows from (\ref{4.9}) and (\ref{4.10}).

(b) Part (b) follows directly from the proof of part (a).
\end{proof}
\begin{rem}
{\rm (a) Theorem \ref{thm4.4} indicates that there exists a finite-horizon optimal deterministic Markov policy in the class of all randomized history-dependent policies.

(b) The existence of a solution to the equality (\ref{4.6}),  the so-called optimality equation for finite-horizon criteria, has been established in \cite{bau,gih,pli,mill,van,w,yush}. More precisely, the transition rates are assumed to be bounded in \cite{bau,mill,gih,pli,yush}, and unbounded in \cite{van,w}. However,  the fact that the finite-horizon optimal value function is the unique solution to the optimality equation  and the existence of optimal policies have not been studied in \cite{van}. The discussions on the existence of optimal policies are restricted to the class of all randomized Markov policies  in the aforementioned works. Theorem  \ref{thm4.4} deals with the finite-horizon criteria in the class of all randomized history-dependent policies, and extends  the results in the previous literature. It should be noted that the controlled state process does not have the Markov property under any randomized history-dependent policy, the extension is nontrivial. Moreover, the fixed point theorem and uniformization techniques are inapplicable to the case of unbounded transition rates.

}
\end{rem}
\section{Linear programming for the constrained case}
In this section we will use the convex analytic approach by introducing the occupation measures of  the finite-horizon criteria to deal with the  constrained case.

Let $w$ be as in Assumption \ref{ass3.1}, and $X:=[0,T]\times K$ is endowed with Borel $\sigma$-algebra $\mathcal{B}(X)$.  $B_{w}(X)$ denotes the set of real-valued measurable functions on $X$ with finite norm $\|g\|_{w}:=\sup_{t\in [0,T]}\sup_{(i,a)\in K}\frac{|g(t,i,a)|}{w(i)}$. Denote by $C_w(X)$ the set of all continuous functions in $B_w(X)$, and by    $\mathcal{P}_w(X)$  the set of all probability measures $\eta$ on $\mathcal{B}(X)$ satisfying $\sum_{i\in S}w(i)\overline{\eta}(i)<\infty$,  where $\overline{\eta}(i):=\eta([0,T],\{i\}, A)$ for all $i\in S$. The set  $\mathcal{P}_w(X)$  is endowed with the $w$-weak topology  for which all mappings $\eta\mapsto\int_{[0,T]}\sum_{i\in S}\int_A g(t,i,a)\eta(dt,i,da)$ are continuous for each $g\in C_w(X)$. Moreover, because $X$ is metrizable,  it follows from Corollary A.44 in \cite{schied} that $\mathcal{P}_w(X)$ is metrizable with respect to the $w$-weak topology.

For each $\pi\in\Pi$, we define the occupation measure of finite-horizon criteria on $\mathcal{B}(X)$ corresponding to $\pi$ by
$$\eta^{\pi}(dt,i, da):=\frac{1}{T}E_{\gamma}^{\pi}\left[I_{\{\xi_{t}=i\}}\pi(da|\omega,t)\right]dt$$
for any $i\in S$. The set of all occupation measures is denoted by $\mathcal{D}$, i.e., $\mathcal{D}:=\{\eta^{\pi}:\ \pi\in\Pi\}$.
Moreover, define
 \begin{eqnarray}\label{5-0}
   \mathcal{D}_0:=\left\{\eta\in \mathcal{D}:\ T\int_{[0,T]}\sum_{i\in S}\int_A c_n(i,a)\eta(dt,i,a)\leq d_n \ {\rm for \ all} \ 1\leq n\leq N\right\}.
 \end{eqnarray}

Before investigating the properties of $\mathcal{D}$, we impose the following condition.
\begin{ass}\label{ass5.1}
\begin{itemize}
\item[{\rm(i)}] For each integer $m\geq1$,  the set $S_m:=\{i\in S: w(i)\leq m\}$ is nonempty and finite, where $w$ is as in Assumption \ref{ass3.1}.
\item[{\rm (ii)}] $\gamma(w^2):=\sum_{i\in S}w^2(i)\gamma(i)<\infty$ and $\gamma(w^3):=\sum_{i\in S}w^3(i)\gamma(i)<\infty$.
\end{itemize}
\end{ass}
\begin{rem}
{\rm
 (a) If $S$ is assumed to be the set of all nonnegative integers and  the function $w$ on $S$ satisfies $\lim\limits_{i\to\infty}w(i)=\infty$, Assumption \ref{ass5.1}(i) holds.

 (b) Assumption \ref{ass5.1}(i) is used to obtain the compactness  of the set of all occupation measures in the $w$-weak topology.
  }
\end{rem}
 Under Assumptions \ref{ass3.1}, \ref{ass3.2}(i),  \ref{ass3.3}(i), and \ref{ass5.1}(ii), Proposition 2.1 in \cite{pi11} gives
\begin{eqnarray}\label{5.1}
\int_0^{T}\sum_{i\in S}\int_Aw^2(i)\eta^{\pi}(dt,i, da)=\frac{1}{T}\int_0^TE_{\gamma}^{\pi}[w^2(\xi_t)]dt\leq e^{\rho_2T}\gamma(w^2)+\frac{b_2}{\rho_2}(e^{\rho_2 T}-1)<\infty
\end{eqnarray}
for all $\pi\in\Pi$. Hence, we have $\mathcal{D}_0\subseteq\mathcal{D}\subseteq \mathcal{P}_{w^2}(X)$.

Now we give  the properties of the occupation measures of finite-horizon criteria below.
\begin{thm}\label{thm5.1}
Suppose that Assumptions \ref{ass3.1}-\ref{ass3.3}, and \ref{ass5.1} hold. Then we have
\begin{itemize}
\item[{\rm(a)}] If $\eta\in \mathcal{P}_{w^2}(X)$, then $\eta\in\mathcal{D}$ if and only if
\begin{eqnarray}\label{5.2}
&&T\int_{[0,T]}\sum_{i\in S}\int_{A}\sum_{j\in S}\int_t^Tg(j,v)dvq(j|i,a)\eta(dt,i,da)\nonumber\\
&=&T\int_{[0,T]}\sum_{i\in S}g(i,t)\eta(dt,i,A)-\int_{[0,T]}\sum_{i\in S}g(i,t)\gamma(i)dt
\end{eqnarray}
for each $g\in C_w(S\times[0,T])$.
\item[{\rm (b)}] The sets $\mathcal{D}$ and $\mathcal{D}_0$ are convex and compact in the $w^2$-weak topology.
\end{itemize}
\end{thm}
\begin{proof}
(a) If $\eta\in \mathcal{D}$, by the definition of $\mathcal{D}$, there exists $\pi\in\Pi$ such that $\eta=\eta^{\pi}$. Fix any $g\in C_w(S\times [0,T])$. Since the function $g$ is $[0,T]$-uniformly $w^2$-bounded,   it follows from Theorem \ref{thm4.3} that
\begin{eqnarray*}
&&T\int_0^T\sum_{i\in S}\int_{A}\sum_{j\in S}\int_t^Tg(j,v)dvq(j|i,a)\eta^{\pi}(dt,i,da)\\
  &=&E_{\gamma}^{\pi}\left[\int_0^T\int_A\sum_{j\in S}\int_t^Tg(j,v)dvq(j|\xi_t,a)\pi(da|\omega, t)dt\right]\\
  &=&E_{\gamma}^{\pi}\left[\int_0^Tg(\xi_t, t)dt\right]-\sum_{i\in S}\int_0^Tg(i,t)dt\gamma(i)\\
  &=&T\int_0^T\sum_{i\in S}g(i,t)\eta^{\pi}(dt,i,A)-\int_0^T\sum_{i\in S}g(i,t)\gamma(i)dt.
\end{eqnarray*}
Hence, $\eta$ satisfies (\ref{5.2}). Conversely, suppose that (\ref{5.2}) holds for some $\eta\in \mathcal{P}_{w^2}(X)$. By Proposition D.8 in \cite[p.184]{one96}, there exists a kernel $\varphi$ on $A$ given $S\times [0,T]$ satisfying $\varphi(A(i)|i,t)=1$ for all $(i,t)\in S\times [0,T]$, and $\eta(dt,i,da)=\eta(dt,i,A)\varphi(da|i,t)$. Let  $\pi^*(\cdot|\ \omega,t):=\varphi(\cdot| \xi_{t-}(\omega), t)$. Below we will show that $\eta=\eta^{\pi^*}$. This is equivalent to proving
\begin{eqnarray}\label{5.3}
\int_{[0,T]}\sum_{i\in S}\int_A h(t,i,a)\eta(dt,i,da)=\int_{[0,T]}\sum_{i\in S}\int_A h(t,i,a)\eta^{\pi^*}(dt,i,da)
\end{eqnarray}
for any bounded  measurable function $h$ on $X$.  Fix any $j,k\in S$, $t\in[0,T]$, and $h\in B(X)$. Define
$$H(j,t):=E_k^{\pi^*}\left[\int_t^{T}\int_Ah(s,\xi_{s},a)\varphi(da|\xi_{s},s)ds\bigg|\ \xi_t=j\right].$$
Then, by Theorem \ref{thm4.2}, we obtain
\begin{align}\label{5.4}
 & H(j,t)=\int_t^T\sum_{i\in S}\int_Ah(s,i,a)\varphi(da|i,s)P_k^{\pi^*}(\xi_s=i|\ \xi_t=j)ds
  =\int_t^T\int_Ah(s,j,a)\varphi(da|j,s)ds\nonumber\\
  &+\int_t^T\sum_{i\in S}\int_Ah(s,i,a)\varphi(da|i,s)\int_t^s\sum_{i'\in S}\int_Aq(i|i',a)\varphi(da|i',v)P_k^{\pi^*}(\xi_v=i'|\ \xi_t=j)dvds.
\end{align}
Let
\begin{align*}
  \mathcal{L}:=\bigg\{g\in B(S\times [t,T]): \ &\int_t^T\sum_{i\in S}g(i,s)\int_t^s\sum_{i'\in S}\int_Aq(i|i',a)\varphi(da|i',v)P_k^{\pi^*}(\xi_v=i'|\ \xi_t=j)dvds\\
  =&\int_t^T\int_A\sum_{i'\in S}\int_s^T\sum_{i\in S}g(i,v)P_k^{\pi^*}(\xi_v=i|\ \xi_s=i')dvq(i'|j,a)\varphi(da|j,s)ds
  \bigg\}.
\end{align*}
Thus, employing the Kolmogorov forward and backward equations and following the similar arguments of Theorem \ref{thm4.3}, we  have
$\mathcal{L}=B(S\times[t,T])$. Since the function $\int_Ah(s,i,a)\varphi(da|i,s)$ is in $B(S\times[t,T])$, we get
\begin{eqnarray*}
  &&\int_t^T\sum_{i\in S}\int_Ah(s,i,a)\varphi(da|i,s)\int_t^s\sum_{i'\in S}\int_Aq(j|i',a)\varphi(da|i',v)P_k^{\pi^*}(\xi_v=i'|\ \xi_t=j)dvds\\
  &=&
  \int_t^T\int_A\sum_{i'\in S}H(i',s)q(i|j,a)\varphi(da|j,s)ds.
\end{eqnarray*}
Hence, combining the last equality and (\ref{5.4}),  we have
\begin{eqnarray*}
  H(j,t)=\int_t^T\int_Ah(s,j,a)\varphi(da|j,s)ds+\int_t^T\int_A\sum_{i\in S}H(i,s)q(i|j,a)\varphi(da|j,s)ds,
\end{eqnarray*}
which implies
\begin{eqnarray}\label{5.5}
  -\frac{\partial H}{\partial t}(j,t)=\int_Ah(t,j,a)\varphi(da|j,t)+\int_A\sum_{i\in S}H(i,t)q(i|j,a)\varphi(da|j,t) \ \ {\rm a.e.} \ t\in [0,T].
\end{eqnarray}
Let $\widetilde{M}:=\sup_{(t,j,a)\in X}|h(t,j,a)|$. Then   it follows from (\ref{5.5}) and Assumption \ref{ass3.1} that
\begin{eqnarray}\label{5.6}
\left|\frac{\partial H}{\partial t}(j,t)\right|\leq (1+2L)\widetilde{M}w(j) \ {\rm for\  all}\  (j,t)\in S\times [0,T].
\end{eqnarray}
Set $g(i,t):=\widetilde{g}(t)$ for all $(i,t)\in S\times [0,T]$ in (\ref{5.2}), where $\widetilde{g}$ is an arbitrary bounded measurable function on $[0,T]$. Then  we conclude that
the marginal of $\eta$ on $[0,T]$ is the normalized Lebesgue measure. Hence, by
Proposition D.8 in \cite[p.184]{one96},  there exists a kernel $\phi$ on $S$ given $[0,T]$ satisfying $\eta(dt,i,A)=\frac{1}{T}\phi(i,t)dt$. Applying the standard technique, we see that (\ref{5.2}) is also true for any function in $B_{w}(S\times [0,T])$.  Therefore,
direct calculations give
\begin{eqnarray*}
  &&\int_{[0,T]}\sum_{i\in S}\int_A h(t,i,a)\eta(dt,i,da)\\
  &=&\frac{1}{T}\int_{[0,T]}\sum_{i\in S}\int_A h(t,i,a)\varphi(da|i,t)\phi(i,t)dt\\
  &=&-\frac{1}{T}\int_{[0,T]}\sum_{i\in S}\left[\frac{\partial H}{\partial t}(i,t)+\int_A\sum_{i'\in S}H(i',t)q(i'|i,a)\varphi(da|i,t)\right]\phi(i,t)dt\\
  &=&-\int_{[0,T]}\sum_{i\in S}\frac{\partial H}{\partial t}(i,t)\eta(dt,i,A)
  +\int_{[0,T]}\sum_{i\in S}\int_A\sum_{i'\in S}\int_t^T\frac{\partial H}{\partial v}(i',v)dvq(i'|j,a)\eta(dt,i,da)\\
  &=&-\frac{1}{T}\int_{[0,T]}\sum_{i\in S}\frac{\partial H}{\partial t}(i,t) \gamma(i)dt\\
  &=&\frac{1}{T}\sum_{i\in S} H(i,0)\gamma(i),
\end{eqnarray*}
where the second equality follows from (\ref{5.5}), and the fourth one is due to (\ref{5.2}) and (\ref{5.6}).  Moreover, it is obvious that $$\int_{[0,T]}\sum_{i\in S}\int_A h(t,i,a)\eta^{\pi^*}(dt,i,da)=\frac{1}{T}\sum_{i\in S} H(i,0)\gamma(i).$$
Thus, the equality (\ref{5.3}) holds. Hence, we obtain $\eta\in\mathcal{D}$.

(b) The convexity of $\mathcal{D}$ and $\mathcal{D}_0$ follows directly  from part (a) and (\ref{5-0}). Let $\{\eta_n\}\subseteq\mathcal{D}$ be an arbitrary sequence converging to a measure $\eta\in\mathcal{P}_{w^2}(X)$ in the $w^2$-weak topology.
 Then it follows from part (a) that
\begin{eqnarray}\label{5.8}
  &&T\int_0^T\sum_{i\in S}\int_{A}\sum_{j\in S}\int_t^Tg(j,v)dvq(j|i,a)\eta_n(dt,i,da)\nonumber\\
&=&T\int_0^T\sum_{i\in S}g(i,t)\eta_n(dt,i,A)-\int_0^T\sum_{i\in S}g(i,t)\gamma(i)dt
\end{eqnarray}
for each $g\in C_w(S\times[0,T])$. By Assumptions \ref{ass3.1} and \ref{ass3.3}(iii), we have that the function $\sum_{j\in S}\int_t^Tg(j,v)dvq(j|i,a)$ belongs to $C_{w^2}(X)$. Thus, combining (\ref{5.8}) and part (a), we obtain $\eta\in \mathcal{D}$. Hence, $\mathcal{D}$ is closed in the
$w^2$-weak topology.

For each integer $m\geq1$, define $X_m:=\{(t,i,a):t\in[0,T], i\in S_m, a\in A(i)\}$ and  $w_m:=\inf\{w(i): (t,i,a)\notin X_m\}$, where $S_m$ is as in Assumption \ref{ass5.1}. Then it follows from Assumption \ref{ass3.3}(ii) that $\{X_m\}$ is a nondecreasing sequence of compact sets, $X_m\uparrow X$, and $\lim\limits_{m\to\infty}w_m=\infty$.  Thus,  we have
\begin{eqnarray*}
  w_m\int_0^{T}\sum_{i\notin S_m}\int_Aw^2(i)\eta^{\pi}(dt,i, da)\leq\frac{1}{T}\int_0^TE_{\gamma}^{\pi}[w^3(\xi_t)]dt\leq e^{\rho_3T}\gamma(w^3)+\frac{b_3}{\rho_3}(e^{\rho_3 T}-1)
\end{eqnarray*}
for all $\pi\in\Pi$. Hence, for any $\epsilon>0$, the last inequality implies that there exists an integer $m_0>0$ satisfying
\begin{eqnarray}\label{5.9}
  \sup_{\pi\in\Pi}\int_0^{T}\sum_{i\notin S_m}\int_Aw^2(i)\eta^{\pi}(dt,i, da)\leq \epsilon.
\end{eqnarray}
Therefore, by (\ref{5.1}),  (\ref{5.9}),  and Corollary A.46 in \cite[p.424]{schied}, we have that $\mathcal{D}$ is  compact in  the
$w^2$-weak topology. Finally, it follows from the compactness of $\mathcal{D}$, (\ref{5-0}),  Assumptions \ref{ass3.2}(ii) and  \ref{ass3.3}(iii)  that $\mathcal{D}_0$ is also compact
in  the $w^2$-weak topology.
\end{proof}
\begin{rem}\label{rem5.2}
{\rm  (a) Theorem \ref{thm5.1} is \emph{new}, and it gives an equivalent characterization of the occupation measures and establishes the compactness and convexity of the set of all occupation measures, which play a crucial role in reformulating the constrained optimization as a linear program and obtaining its dual program.

(b) From the proof of part (a), we conclude that
 for each occupation measure generated by a randomized history-dependent policy, there exists an occupation measure generated by a randomized Markov policy  equal to it. Hence,
  for any $\pi\in\Pi$, there exists $\widetilde{\pi}\in\Pi^M$ such that $V_n(\pi)=V_n(\widetilde{\pi})$ for all $1\leq n\leq N$.

}
\end{rem}
The constrained optimization problem (\ref{P}) is equivalent to the linear programming formulation below:
\begin{eqnarray}\label{5.10}
  &&\text{minimize} \ \ \ \   T\int_0^T\sum_{i\in S}\int_Ac_0(i,a)\eta(dt,i,da)\ \text{over} \ \eta\in\mathcal{D} \nonumber\\
  &&\text{subject \ to} \ \  T\int_0^T\sum_{i\in S}\int_Ac_n(i,a)\eta(dt,i,da)\leq d_n, \ n=1,2,\ldots, N.
\end{eqnarray}

The following statement establishes the existence of constrained-optimal policies for the case of unbounded transition and cost rates.
\begin{thm}\label{thm5.2}
Suppose that Assumptions \ref{ass3.1}-\ref{ass3.3} and \ref{ass5.1} are satisfied. Then there exists a constrained-optimal policy $\pi^*\in \Pi^M$ for the constrained optimization problem {\rm (\ref{P})}.
\end{thm}
\begin{proof}
It follows from the proof of part (a) in Theorem \ref{thm5.1}  that $\mathcal{D}=\{\eta^{\pi}: \pi\in\Pi^M\}$. Hence, the desired  assertion follows from (\ref{5.10}), Assumptions \ref{ass3.2}(ii), \ref{ass3.3}(iii), Theorem \ref{thm5.1},  and the Weierstrass theorem in \cite[p.40]{al}.
\end{proof}
Below we will develop the dual program of the linear program (\ref{5.10}), and provide the strong duality conditions. To this end, we introduce the following notation.

Let $w$ be as in Assumption \ref{ass3.1}, and $Y:=[0,T]\times S$ is endowed with Borel $\sigma$-algebra $\mathcal{B}(Y)$. We denote by $\mathbb{R}$ the space of all real numbers (i.e., $\mathbb{R}:=(-\infty,\infty)$), by $\mathcal{M}_{w^3}(X)$ the space of all signed measures $\eta$ on $\mathcal{B}(X)$ with $\int_{[0,T]}\sum_{i\in S}\int_Aw^3(i)|\eta|(dt,i,da)<\infty$, and by
$\mathcal{M}^+_{w^3}(X):=\{\eta\in\mathcal{M}_{w^3}(X):\eta\geq0\}$, where $|\eta|:=\eta^++\eta^-$ denotes the total variation of $\eta$.  $\mathcal{M}_{w^2}(Y)$ and $\mathcal{M}^+_{w^2}(Y)$ are defined similarly. Moreover, let
\begin{eqnarray*}
  &&\mathcal{X}:=\mathcal{M}_{w^3}(X)\times \mathbb{R}^N, \ \ \mathcal{Y}:=B_{w^3}(X)\times \mathbb{R}^N,\\
  &&\mathcal{Z}:=\mathcal{M}_{w^2}(Y)\times\mathbb{R}^N, \ \ \mathcal{U}:=B_{w^2}(Y)\times \mathbb{R}^N.
\end{eqnarray*}
Define a bilinear map $\langle\rangle_1$ on the dual pair of $(\mathcal{X},\mathcal{Y})$ by
\begin{eqnarray}\label{5.11}
\langle(\eta, x_1,\ldots,x_N), (g, y_1,\ldots,y_N)\rangle_1:=\int_{[0,T]}\sum_{i\in S}\int_A g(t,i,a)\eta(dt,i,da)+\sum_{n=1}^{N}x_ny_n
\end{eqnarray}
for all $(\eta,x_1,\ldots,x_N)\in \mathcal{X}$ and $(g,y_1,\ldots, y_N)\in\mathcal{Y}$, and another bilinear  $\langle\rangle_2$ on the dual pair of $(\mathcal{Z},\mathcal{U})$ by
\begin{eqnarray}\label{5.12}
\langle(\nu, z_1,\ldots,z_N), (h, u_1,\ldots,u_N)\rangle_2:=\int_{[0,T]}\sum_{i\in S}h(t,i)\nu(dt,i)+\sum_{n=1}^{N}z_nu_n
\end{eqnarray}
for all $(\nu,z_1,\ldots,z_N)\in \mathcal{Z}$ and $(h,u_1,\ldots, u_N)\in\mathcal{U}$. Moreover, two operators  $\Gamma$ from $\mathcal{X}$ to $\mathcal{Z}$, and $\Gamma^*$ from   $\mathcal{U}$ to $\mathcal{Y}$ are defined as follows:
\begin{eqnarray*}
  &&\Gamma (\eta, x_1,\ldots,x_N):=\bigg(\eta(ds,j,A)-\int_{[0,T]}\sum_{i\in S}\int_AI_{[t,T]}(s)q(j|i,a)\eta(dt,i,da)ds,\nonumber\\
  &&T\int_{[0,T]}\sum_{i\in S}\int_Ac_1(i,a)\eta(dt,i,da)+x_1,\ldots,T\int_{[0,T]}\sum_{i\in S}\int_Ac_N(i,a)\eta(dt,i,da)+x_N \bigg),\nonumber\\
\end{eqnarray*}
and
\begin{eqnarray*}
  \Gamma^*(h,u_1,\ldots,u_N):=\!\bigg(h(t,i)-\int_t^T\sum_{j\in S} h(s,j)q(j|i,a)ds+T\sum_{n=1}^{N}u_nc_n(i,a), u_1,\ldots, u_N\bigg).
\end{eqnarray*}

Then  we have the following lemma on the properties of $\Gamma$ and $\Gamma^*$.
\begin{lem}\label{lem5.1}
  Under Assumptions \ref{ass3.1}-\ref{ass3.3}, and \ref{ass5.1}, the following statements hold.
  \begin{itemize}
    \item[{\rm (a)}] $\Gamma\mathcal{X}\subseteq \mathcal{Z}$ and $\Gamma^*(\mathcal{U})\subseteq\mathcal{Y}$.
    \item[{\rm (b)}] $\Gamma^*$ is the adjoint of  $\Gamma$.
    \item[{\rm (c)}] $\Gamma$ is $\tau(\mathcal{X},\mathcal{Y})-\tau(\mathcal{Z},\mathcal{U})$ continuous, where $\tau(\mathcal{X},\mathcal{Y})$ denotes the coarsest topology on $\mathcal{X}$ such that $\langle\cdot, y\rangle_1$ is continuous on $\mathcal{X}$ for each $y\in\mathcal{Y}$, and $\tau(\mathcal{Z},\mathcal{U})$ is defined similarly.
  \end{itemize}
\end{lem}
\begin{proof}
  (a) For each $(\eta, x_1,\ldots,x_N)\in\mathcal{X}$, it follows from Assumption \ref{ass3.2}(ii) that
  $$\int_{[0,T]}\sum_{i\in S}\int_A|c_n(i,a)||\eta|(dt,i,da)\leq M\int_{[0,T]}\sum_{i\in S}\int_Aw(i)|\eta|(dt,i,da)<\infty$$
  for all $1\leq n\leq N$. Moreover, direct calculations yield
  \begin{eqnarray*}
    &&\int_{[0,T]}\sum_{i\in S} w^2(i)|\eta|(dt,i,A)+\int_{[0,T]}\sum_{j\in S}w^2(j)\int_{[0,T]}\sum_{i\in S}\int_AI_{[t,T]}(s)ds|q(j|i,a)||\eta|(dt,i,da)\\
    &\leq&[1+T(\rho_2+b_2+2L)]\int_{[0,T]}\sum_{i\in S} w^3(i)|\eta|(dt,i,A)<\infty.
  \end{eqnarray*}
  Hence, we obtain $\Gamma\mathcal{X}\subseteq \mathcal{Z}$.

For each $(h,u_1,\ldots,u_N)\in \mathcal{U}$, set $\overline{L}:=\sup_{(t,i)\in Y}\frac{|h(t,i)|}{w^2(i)}$. Then, by Assumptions \ref{ass3.1}(ii),  \ref{ass3.2}(ii), and \ref{ass3.3}(i),  we have
\begin{eqnarray*}
  &&\sup_{(t,i,a)\in X}\frac{\left|h(t,i)-\int_t^T\sum_{j\in S} h(s,j)q(j|i,a)ds+\sum_{n=1}^{N}u_nc_n(i,a)\right|}{w^3(i)}\\
  &\leq&\overline{L}[1+T(\rho_2+b_2+2L)]+M\sum_{n=1}^{N}|u_n|,
\end{eqnarray*}
which implies $\Gamma^*(\mathcal{U})\subseteq\mathcal{Y}$.

(b) For each $(\eta, x_1,\ldots,x_N)\in\mathcal{X}$ and $(h,u_1,\ldots,u_N)\in \mathcal{U}$, it follows from (\ref{5.11}), (\ref{5.12}), and the definitions of $\Gamma$ and $\Gamma^*$ that
\begin{eqnarray*}
  &&\langle\Gamma(\eta, x_1,\ldots,x_N), (h,u_1,\ldots,u_N)\rangle_2\\
  &=&\int_{[0,T]}\sum_{i\in S}h(t,i)\eta(dt,i,A)-\int_{[0,T]}\sum_{j\in S}h(s,j)\int_{[0,T]}\sum_{i\in S}\int_AI_{[t,T]}(s)dsq(j|i,a)\eta(dt,i,da)\\
  &&+\sum_{n=1}^{N}\left[T\int_{[0,T]}\sum_{i\in S}\int_Ac_n(i,a)\eta(dt,i,da)+x_n\right]u_n\\
  &=&\int_{[0,T]}\sum_{i\in S}h(t,i)\eta(dt,i,A)-\int_{[0,T]}\sum_{i\in S}\int_A\int_t^T\sum_{j\in S}h(s,j)q(j|i,a)ds\eta(dt,i,da)\\
  &&+\sum_{n=1}^{N}Tu_n\int_{[0,T]}\sum_{i\in S}\int_Ac_n(i,a)\eta(dt,i,da)+\sum_{n=1}^{N}x_nu_n\\
  &=&\langle(\eta, x_1,\ldots,x_N), \Gamma^*(h,u_1,\ldots,u_N)\rangle_1.
\end{eqnarray*}
Therefore, we get the desired assertion.

(c) Part (c) follows from part (a) and Proposition 12.2.5 in \cite[p.208]{one99}.
\end{proof}
According to Lemma \ref{lem5.1} and Chapter 12 in \cite{one99}, the constrained problem (\ref{5.10}) can be rewritten as
\begin{eqnarray}\label{5.13}
 \textbf{P}: &&\text{minimize} \ \ \ \langle(\eta, x_1,\ldots, x_N), (Tc_0,0,\ldots,0)\rangle_1\nonumber\\
  &&\text{subject\ to} \ \ \Gamma(\eta, x_1,\ldots, x_N)=(\frac{1}{T}\gamma dt, d_1,\ldots, d_N )\nonumber\\
  && \ \ \ \ \ \ \ \ \ \ \ \ \ \  \ \eta\in\mathcal{M}^+_{w^3}(X), \ x_1\geq0, \ldots, x_N\geq0.
\end{eqnarray}
The corresponding dual problem of \textbf{P} is
\begin{eqnarray}\label{5.14}
\textbf{P}^*: &&\text{maximize} \ \ \ \langle(\frac{1}{T}\gamma dt, d_1,\ldots, d_N), (h,u_1,\ldots,u_N)\rangle_2\nonumber\\
&&\text{subject\ to} \ \ Tc_0(i,a)-h(t,i)+\int_t^T\sum_{j\in S}h(s,j)q(j|i,a)ds-T\sum_{n=1}^{N}u_nc_n(i,a)\geq0\nonumber\\
&&\ \ \ \ \ \ \ \ \ \ \ \ \ \  \ \text{for \ all} \ (t,i,a)\in X, \ h\in B_{w^2}(Y), \  u_1\leq0,\ldots, u_N\leq0.
\end{eqnarray}

We denote the values of problems (\ref{5.13}) and (\ref{5.14}) by $\inf(\textbf{P})$ and $\sup(\textbf{P}^*)$, respectively.

In order to establish the strong duality between the primal linear program (\ref{5.13}) and its dual program (\ref{5.14}), we impose the following Slater condition which is commonly used in the constrained optimization problems; see, for instance, \cite{guo13,guo12,guo11,pi11}.
\begin{ass}\label{ass5.2}
There exists a policy $\widetilde{\pi}\in\Pi$ such that $V_n(\widetilde{\pi})<d_n$ for all $1\leq n\leq N$.
\end{ass}
Now we present the strong duality theorem for finite-horizon criteria below.
\begin{thm}\label{thm5.3}
Suppose that Assumptions \ref{ass3.1}-\ref{ass3.3}, \ref{ass5.1}, and \ref{ass5.2} hold. Then problems {\rm(\ref{5.13})} and {\rm(\ref{5.14})} admit optimal solutions, and $\inf(\textbf{P})=\sup(\textbf{P}^*)$.
\end{thm}
\begin{proof}
By Theorem \ref{thm5.2}, we obtain that the problem {\rm(\ref{5.13})} admits an optimal solution. Below we will show the existence of optimal solutions for  the problem {\rm(\ref{5.14})}.
It follows from Theorem \ref{thm5.1}  that $\mathcal{D}$ is convex. By
  Assumption \ref{ass5.2}, Theorem 17 and Example 1 in \cite[p.7, 18, 23]{rock},   there exist constants  $u^*_n\geq0$ $(1\leq n\leq N)$ such that
  \begin{align}\label{5.15}
    \inf(\textbf{P})=&\sup_{u_n\geq 0, 1\leq n\leq N}\inf_{\eta\in\mathcal{D}}\left\{T\int_{[0,T]}\sum_{i\in S}\int_A\left(c_0(i,a)+\sum_{n=1}^{N}u_nc_n(i,a)\right)\eta(dt,i,da)-\sum_{n=1}^{N}u_nd_n\right\}\nonumber\\
    =&\inf_{\eta\in\mathcal{D}}\left\{T\int_{[0,T]}\sum_{i\in S}\int_A\left(c_0(i,a)+\sum_{n=1}^{N}u^*_nc_n(i,a)\right)\eta(dt,i,da)-\sum_{n=1}^{N}u^*_nd_n\right\}.
  \end{align}
Moreover, when $c_0(i,a)$ in (\ref{4.6})  is replaced by $T(c_0(i,a)+\sum_{n=1}^{N}u^*_nc_n(i,a))-\sum_{n=1}^{N}u^*_nd_n$, following the similar arguments of Theorem \ref{thm4.4}, we conclude that there exists a function $h^*\in B_{w}(Y)$ such that for each $i\in S$ and $t\in[0,T]$,
\begin{eqnarray*}
h^*(t,i)=\int_t^T\inf_{a\in A(i)}\left\{Tc_0(i,a)+T\sum_{n=1}^{N}u^*_nc_n(i,a)-\sum_{n=1}^{N}u^*_nd_n+\sum_{j\in S}h^*(s,j)q(j|i,a)\right\}ds,
\end{eqnarray*}
which implies
\begin{align}\label{5.16}
  -\frac{\partial h^*}{\partial t}(t,i)=\inf_{a\in A(i)}\left\{Tc_0(i,a)+T\sum_{n=1}^{N}u^*_nc_n(i,a)-\sum_{n=1}^{N}u^*_nd_n+\sum_{j\in S}h^*(t,j)q(j|i,a)\right\}.
\end{align}
 Let $\widetilde{h}(t,i):=\sum_{n=1}^{N}u_n^*d_n-\frac{\partial h^*}{\partial t}(t,i)$ for all $(t,i)\in Y$. Then, by (\ref{5.16}), we have
 \begin{eqnarray*}
   Tc_0(i,a)-T\sum_{n=1}^{N}(-u^*_n)c_n(i,a)-\widetilde{h}(t,i)
   +\int_t^T\sum_{j\in S}\widetilde{h}(s,j)q(j|i,a)ds\geq0
 \end{eqnarray*}
 for all $(t,i,a)\in X$. As in the proof of Theorem \ref{thm4.4}, we obtain that $\widetilde{h}\in B_{w^2}(Y)$, and
 \begin{align}\label{5.17}
 \sum_{i\in S}h^*(0,i)\gamma(i)=&\inf_{\pi\in\Pi}E_{\gamma}^{\pi}\left[\int_0^T\int_A\left(Tc_0(\xi_t,a)
 +T\sum_{n=1}^{N}u_n^*c_n(\xi_t,a)-\sum_{n=1}^{N}u^*_nd_n\right)\pi(da|\omega,t)dt\right]\nonumber\\
 =&T\inf(\textbf{P}),
 \end{align}
 where the second equality is due to (\ref{5.15}).
 Hence, $(\widetilde{h},-u_1^*, \ldots, -u_n^*)$ is feasible for $\textbf{P}^*$.
 Direct calculations give
 \begin{eqnarray}\label{5.18}
  \sup(\textbf{P}^*)&\geq&\langle(\frac{1}{T}\gamma dt, d_1,\ldots, d_N), (\widetilde{h},-u^*_1,\ldots,-u^*_N)\rangle_2\nonumber\\
  &=&\frac{1}{T}\int_0^T\sum_{i\in S}\widetilde{h}(t,i)\gamma (i)dt-\sum_{n=1}^{N}u_n^*d_n\nonumber\\
   &=&-\frac{1}{T}\int_0^T\sum_{i\in S}\frac{\partial h^*}{\partial t}(t,i)\gamma(i)dt\nonumber\\
   &=&\frac{1}{T}\sum_{i\in S}h^*(0,i)\gamma(i)=\inf(\textbf{P}).
   \end{eqnarray}
   where the last equality follows from (\ref{5.17}). Moreover, by Theorem 12.2.9 in \cite[p.213]{one99} and Theorem \ref{thm5.2}, we have $\sup(\textbf{P}^*)\leq\inf(\textbf{P})$. Hence, this inequality and (\ref{5.18}) gives $\inf(\textbf{P})=\sup(\textbf{P}^*)$, which implies that
  $(\widetilde{h},-u_1^*, \ldots, -u_n^*)$ is an optimal solution for $\textbf{P}^*$.
\end{proof}
\section{An Example}
In this section, we will show the applications of finite-horizon expected total cost criteria with
a controlled birth and death system,  which has been used to illustrate the existence of average constrained-optimal policies in \cite{guo12}.
\begin{exm}\label{exm6.1}
{\rm  (A controlled birth and death system in \cite{guo12}).  The control model is given as follows: $S:=\{0,1,2,\ldots\}$, $A(0):=[-\lambda,\lambda]\times\{0\}$, $A(i):=[-\lambda,\lambda]\times [-\mu,\mu]$ for all $i\geq1$, $q(1|0,(a_1,0))=-q(0|0,(a_1,0)):=\lambda+a_1$ for all $a_1\in [-\lambda,\lambda]$, and for each $i\geq1$,  $a=(a_1,a_2)\in A(i)$,
\begin{eqnarray*}
q(j|i,a):=\left\{\begin{array}{ll}
\lambda i+a_1, &\textrm{if $j=i+1$},\\
-(\mu+\lambda)i-a_1-a_2, &\textrm{if $j=i$},\\
\mu i+a_2, &\textrm{if $j=i-1$},\\
0, &\text{otherwise},
\end{array}\right.
\end{eqnarray*}
where positive constants $\lambda$ and $\mu$ denote the birth and death rates, respectively.

To ensure the existence of optimal policies for the unconstrained and constrained cases, we consider the following conditions.
\begin{itemize}
\item[(C1)] There exists a positive constant $M$ such that $|c_n(i,a)|\leq M(i+1)$ for all $(i,a)\in K$ and $0\leq n\leq N$.
\item[(C2)] For each fixed $i\in S$ and $n\in\{0,1,\ldots,N\}$, $c_n(i,a)$ is continuous in $a\in A(i)$.
\item[(C3)] There exists  a kernel $\widehat{\varphi}$ on $A$ given $S\times [0,T]$ such that $\int_{A(i)}c_n(i,a)\widehat{\varphi}(da|i,t)<d_n$ for all $(i,t)\in S\times [0,T]$ and $1\leq n\leq N$.
\item[(C4)] The initial distribution $\gamma$ on $S$ satisfies $\sum_{i\in S}i^3\gamma(i)<\infty$.
\end{itemize}
}
\end{exm}
\begin{prop}\label{prop6.1}
 Under conditions {\rm(C1)} and {\rm(C2)}, the controlled birth and death system above satisfies Assumptions \ref{ass3.1}, \ref{ass3.2}(ii), and \ref{ass3.3}. Hence, (by Theorem \ref{thm4.4}), there exists a finite-horizon optimal policy. Moreover, under conditions {\rm(C1)}-{\rm(C4)}, the controlled birth and death system satisfies Assumptions \ref{ass3.1}-\ref{ass3.3}, \ref{ass5.1}, \ref{ass5.2}. Hence, (by Theorem \ref{thm5.3}), there exists a constrained-optimal policy and the strong duality holds.
\end{prop}
\begin{proof}
Let $w(i):=i+1$ for all $i\in S$.
For $i=0$ and $a=(a_1,a_2)\in A(0)$, we have
\begin{eqnarray}
q^*(0)&\leq&2\lambda w(0),\nonumber\\
  \sum_{j\in S}w(j)q(j|0,a)&=&\lambda+a_1\leq \lambda w(0)+\lambda, \label{6.1}\\
  \sum_{j\in S}w^2(j)q(j|0,a)&=&3(\lambda+a_1)\leq 3\lambda w^2(0)+3\lambda,\nonumber\\
   \sum_{j\in S}w^3(j)q(j|0,a)&=&7(\lambda+a_1)\leq 7\lambda w^3(0)+7\lambda.\nonumber
\end{eqnarray}
For each $i\geq1$ and $a=(a_1,a_2)\in A(i)$, straightforward  calculations give
\begin{eqnarray}
q^*(i)&\leq&(\lambda+\mu)i+\lambda+\mu=(\lambda+\mu)w(i),\nonumber\\
\sum_{j\in S}w(j)q(j|i,a)&=&(\lambda-\mu)i+a_1-a_2\leq (\lambda+\mu)w(i),\label{6.2}\\
\sum_{j\in S}w^2(j)q(j|i,a)&=&2(\lambda-\mu)i^2+(3\lambda-\mu+2a_1-2a_2)i+3a_1-a_2\nonumber\\
&\leq& (7\lambda+5\mu)w^2(i)+3\lambda+\mu,\nonumber\\
\sum_{j\in S}w^3(j)q(j|i,a)&=&3(\lambda-\mu)i^3+(9\lambda-3\mu+3a_1-3a_2)i^2\nonumber\\
&&+(7\lambda-\mu+9a_1-3a_2)i+7a_1-a_2\nonumber\\
&\leq&(31\lambda+13\mu)w^3(i)+7\lambda+\mu.\nonumber
\end{eqnarray}
Hence, Assumptions \ref{ass3.1} and \ref{ass3.3}(i) are satisfied  with $L:=2\lambda+\mu$,  $\rho_1:=\lambda+\mu$, $b_1:=\lambda$, $\rho_2:=7\lambda+5\mu$, $b_2:=3\lambda+\mu$, $\rho_3:=31\lambda+13\mu$, and  $b_3:=7\lambda+\mu$.
Moreover, by the description of the model, (\ref{6.1}), (\ref{6.2}),   and conditions (C1)-(C4), we see that Assumptions \ref{ass3.2}, \ref{ass3.3}(ii)(iii), \ref{ass5.1}(ii), and \ref{ass5.2} hold. Finally, for each integer $m\geq1$, we obtain  $S_m=\{0,1,\ldots,m-1\}$ which implies Assumption \ref{ass5.1}(i). This completes the proof of the proposition.
\end{proof}
\begin{rem}
{\rm (a) Since the finite-horizon optimality criterion  is different from the average optimality criterion discussed in \cite{guo12}, the conditions used in Example \ref{exm6.1}  also differ from those in \cite{guo12}. For example, the condition ``$\lambda>\mu$"   in \cite{guo12} is not required
for Example \ref{exm6.1}, whereas there is no need to impose condition (C4) in \cite{guo12}.

(b) The transition rates in \cite{bau,gih,pli,yush} are assumed to be bounded. Therefore, the conditions in the aforementioned works are inapplicable to Example \ref{exm6.1} because it allows the transition  rates to be unbounded from above and from below.

}
\end{rem}
\section*{Acknowledgements}
The research was supported by NSFC.


\begin{thebibliography}{00}
\bibitem{al} Aliprantis, C., \&  Border, K. (2007). Infinite dimensional analysis.  New York: Springer.
\bibitem{bau} B\"{a}uerle, N., \&  Rieder, U. (2011).   Markov decision processes with applications to finance. Berlin:    Springer.
\bibitem{van} van Dijk,  N. M. (1988).    On the finite horizon Bellman equation for controlled Markov jump models with unbounded characteristics: existence and approximation.    Stochastic Processes Appl.,    28,   141-157.
\bibitem{dyn} Dynkin, E. B., \& Yushkevich, A. A. (1979).   Controlled Markov processes.  New York: Springer.

\bibitem{schied}   F\"{o}llmer, H.,  \&   Schied, A. (2004). Stochastic finance: an introduction in discrete time.   Berlin:  Walter de Gruyter.
\bibitem{gih} Gihman, I. I., \&  Skohorod, A. V. (1979).    Controlled stochastic processes.  Berlin:   Springer.
\bibitem{guo09} Guo, X. P., \&   Hern\'{a}ndez-Lerma, O. (2009).  Continuous-time Markov decision processes: theory and applications. Berlin:  Springer.
\bibitem{guo11} Guo, X. P., \&  Piunovskiy, A. (2011).  Discounted continuous-time Markov decision processes with constraints: unbounded transition and loss rates.   Math. Oper. Res.,  36,   105-132.
\bibitem{gs11} Guo, X. P., \&  Song, X. Y. (2011).  Discounted continuous-time constrained Markov decision processes in Polish spaces.  Ann. Appl. Probab.,  21,  2016-2049.
\bibitem{guo12} Guo, X. P.,  Huang, Y. H., \&  Song, X. Y. (2012).   Linear Programming and constrained average  optimality for general continuous-time Markov
decision processes in history-dependent policies.    SIAM J. Control Optim.,   50,   23-47.
\bibitem{guo13} Guo,  X. P.,  Vykertas, M.,  \&   Zhang Y. (2013).   Absorbing continuous-time Makov decision processes with total cost criteria.  Adv. Appl. Probab.,  45,  490-519.
\bibitem{one96}    Hern\'{a}ndez-Lerma, O., \&  Lasserre, J. B. (1996).  Discrete-time Markov control processes: basic optimality criteria. New York:  Springer.
\bibitem{one99}   Hern\'{a}ndez-Lerma, O., \&   Lasserre,  J. B. (1999).  Further topics on discrete-time Markov control processes.  New York: Springer.
\bibitem{jacod} Jacod, J. (1975).  Multivariate point processes: predictable projection, Radon-Nykodym derivatives, representation of martingales.  Z. Wahrscheinlichkeitstheorie verw. Gebite. 31,   235-253.
\bibitem{Kall} Kallenberg,  O. (2002). Foundations of modern probability.  New York:  Springer.
\bibitem{kit85} Kitaev, M. Y. (1985).  Semi-Markov and jump Markov controlled models: average cost criterion.  Theory Probab. Appl.,  30,   272-288.
\bibitem{kit95} Kitaev, M. Y.,  \& Rykov, V. V. (1995).  Controlled queueing systems. Boca Raton:   CRC Press.
\bibitem{mill} Miller, B. L. (1968).   Finite state continuous time Markov decision processes with finite planning horizon.    SIAM J. Control,   6,    266-280.
\bibitem{pi11} Piunovskiy, A.,   \& Zhang, Y. (2011).  Discounted continuous-time Markov decision processes with unbounded rates: the convex analytic approach.  SIAM J. Control Optim.,  49,    2032-2061.
\bibitem{pli} Pliska, S. R. (1975).   Controlled jump processes.   Stochastic Processes Appl.,   3,     259-282.
\bibitem{put}   Puterman,  M. L.  (1994). Markov decision processes: discrete stochastic dynamic programming. New York:   Wiley.
\bibitem{rock} Rockafellar, R.T. (1974).   Conjugate duality and optimization.  Philadelphia:  SIAM.
\bibitem{w}  Wei, Q. D.,   \&    Chen, X. (2014).   Strong average optimality criterion for continuous-time Markov decision processes.    Kybernetika, to appear.
\bibitem{yush}    Yushkevich,  A. A. (1980).   Controlled jump Markov models.     Theory Probab. Appl.,  25,     244-266.

\end{thebibliography}
\end{document}